\documentclass[11pt]{article}
\usepackage{amsmath}
\usepackage{amssymb}
\usepackage{amsthm}
\newtheorem{theorem}{Theorem}[section]
\newtheorem{lemma}[theorem]{Lemma}
\newtheorem{corollary}[theorem]{Corollary}
\newtheorem{proposition}[theorem]{Proposition}
\newtheorem{definition}[theorem]{Definition}
\newtheorem{example}[theorem]{Example}
\newtheorem{question}[theorem]{Question}

\newtheorem*{thmLM}{Theorem LM}
\newtheorem*{thmS}{Theorem S}
\newtheorem*{thmSs}{Theorem S$'$}
\newtheorem*{thmAG}{Theorem AG}
\newtheorem*{thmCMP}{Theorem CMP}
\newtheorem*{thmCS}{Theorem CS}
\newtheorem*{exBM}{Example BM}
\newtheorem*{thmCM}{Theorem CM}
\newtheorem*{qA}{Question A}
\newtheorem*{qBG}{Question BG}
\newtheorem*{qGS}{Question GS}
\newtheorem*{thmMF}{Theorem MF}

\numberwithin{equation}{section}

\begin{document}

\def\vru{\vrule width0pt height0pt depth12pt}
\def\C{{\mathbb C}}
\def\N{{\mathbb N}}
\def\Z{{\mathbb Z}}
\def\R{{\mathbb R}}
\def\K{{\mathbb K}}
\def\T{{\mathbb T}}
\def\D{{\mathbb D}}
\def\P{{\mathbb P}}
\def\G{{\mathbb G}}
\def\F{{\cal F}}
\def\U{{\cal U}}
\def\M{{\cal M}}
\def\L{{\cal L}}
\def\H{{\cal H}}
\def\CC{{\cal C}}
\def\E{{\cal E}}
\def\rr{{\cal R}}
\def\pp{{\cal P}}
\def\epsilon{\varepsilon}
\def\kappa{\varkappa}
\def\phi{\varphi}
\def\leq{\leqslant}
\def\geq{\geqslant}
\def\re{\text{\tt Re}\,}
\def\slim{\mathop{\hbox{$\overline{\hbox{\rm lim}}$}}\limits}
\def\ilim{\mathop{\hbox{$\underline{\hbox{\rm lim}}$}}\limits}
\def\supp{\hbox{\tt supp}\,}
\def\dim{\hbox{\tt dim}\,}
\def\ker{\hbox{\tt ker}\,}
\def\spann{\hbox{\tt span}\,}
\def\Re{\hbox{\tt Re}\,}
\def\ssub#1#2{#1_{{}_{{\scriptstyle #2}}}}
\font\Goth=eufm10 scaled 1200
\def\uu{\hbox{{\Goth U}}}
\def\uuu{\hbox{{\Goth H}}}
\def\bin#1#2{\left({{#1}\atop {#2}}\right)}
\def\ww{\widetilde w}

\title{Remarks on common hypercyclic vectors}

\author{Stanislav Shkarin}

\date{}

\maketitle

\begin{abstract} We treat the question of existence of common
hypercyclic vectors for families of continuous linear operators. It
is shown that for any continuous linear operator $T$ on a complex
Fr\'echet space $X$ and a set $\Lambda\subseteq \R_+\times\C$ which
is not of zero three-dimensional Lebesgue measure, the family
$\{aT+bI:(a,b)\in\Lambda\}$ has no common hypercyclic vectors. This
allows to answer negatively questions raised by Godefroy and Shapiro
and by Aron. We also prove a sufficient condition for a family of
scalar multiples of a given operator on a complex Fr\'echet space to
have a common hypercyclic vector. It allows to show that if
$\D=\{z\in\C:|z|<1\}$ and $\phi\in \H^\infty(\D)$ is non-constant,
then the family $\{zM_\phi^\star:b^{-1}<|z|<a^{-1}\}$ has a common
hypercyclic vector, where $M_\phi:\H^2(\D)\to \H^2(\D)$, $M_\phi
f=\phi f$, $a=\inf\{|\phi(z)|:z\in\D\}$ and
$b=\sup\{|\phi(z)|:|z|\in\D\}$, providing an affirmative answer to a
question by Bayart and Grivaux. Finally, extending a result of
Costakis and Sambarino, we prove that the family
$\{aT_b:a,b\in\C\setminus\{0\}\}$ has a common hypercyclic vector,
where $T_bf(z)=f(z-b)$ acts on the Fr\'echet space $\H(\C)$ of
entire functions on one complex variable.
\end{abstract}



\section{Introduction \label{s1}}\rm

All vector spaces in this article are assumed to be over $\K$ being
either the field $\C$ of complex numbers or the field $\R$ of real
numbers. Throughout this paper all topological spaces and
topological vector spaces {\bf are assumed to be Huasdorff}. As
usual, $\Z_+$ is the set of non-negative integers, $\R_+$ is the set
of non-negative real numbers, $\N$ is the set of positive integers,
$\K^\star=\K\setminus\{0\}$, $\D=\{z\in\C:|z|<1\}$ and
$\T=\{z\in\C:|z|=1\}$. By a {\it compact interval} of the real line
we mean a set of the shape $[a,b]$ with $-\infty<a<b<\infty$. That
is, a singleton is {\bf not} considered to be an interval. For
topological vector spaces $X$ and $Y$, $L(X,Y)$ stands for the space
of continuous linear operators from $X$ to $Y$. We write $L(X)$
instead of $L(X,X)$ and $X^*$ instead of $L(X,\K)$. For $T\in
L(X,Y)$, the dual operator $T^*:Y^*\to X^*$ acts according to the
formula $T^*f(x)=f(Tx)$. Recall \cite{sch} that an $\F$-space is a
complete metrizable topological vector space and a Fr\'echet space
is a locally convex $\F$-space. For a subset $A$ of a vector space
$X$, symbol $\spann(A)$ stands for the linear span of $A$.

\begin{definition}\label{def1}\rm
Let $X$ and $Y$ be topological spaces and $\F=\{T_a:a\in A\}$ be a
family of continuous maps from $X$ to $Y$. An element $x\in X$ is
called {\it universal} for $\F$ if the orbit $\{T_ax:a\in A\}$ is
dense in $Y$ and $\F$ is said to be {\it universal} if it has a
universal element. We denote the set of universal elements for $\F$
by the symbol $\uu(\F)$. A continuous linear operator $T$ acting on
a topological vector space $X$ is called {\it hypercyclic} if the
family of its powers $\{T^n:n\in \Z_+\}$ is universal. Corresponding
universal elements are called {\it hypercyclic vectors} for $T$. The
set of hypercyclic vectors for $T$ is denoted by $\uuu(T)$. That is,
$\uuu(T)=\uu(\{T^n:n\in\Z_+\})$. If $\{T_a:a\in A\}$ is  a family of
continuous linear operators on topological vector space $X$, we
denote
$$
\uuu\{T_a:a\in A\}=\bigcap_{a\in A}\uuu(T_a).
$$
That is, $\uuu\{T_a:a\in A\}$ consists of all vectors $x\in X$ that
are hypercyclic for each $T_a$, $a\in A$.
\end{definition}

Recall that a topological space $X$ is called {\it Baire} if the
intersection of any countable family of dense open subsets of $X$ is
dense. Hypercyclic operators and universal families have been
intensely studied during last few decades, see surveys
\cite{ge1,ge2} and references therein. It is well-known \cite{ge1}
that the set of hypercyclic vectors of a hypercyclic operator on a
separable metrizable Baire topological vector space is a dense
$G_\delta$-set. It immediately follows that any countable family of
hypercyclic operators on such a space has a dense $G_\delta$-set of
common hypercyclic vectors (=hypercyclic for each member of the
family). We are interested in the existence of common hypercyclic
vectors for uncountable families of continuous linear operators.
First results in this direction were obtained by Abakumov and Gordon
\cite{ag} and L\'eon--Saavedra and M\"uller \cite{muller}.

\begin{thmAG}Let $T$ be the backward shift on $\ell_2$. That is, $T\in
L(\ell_2)$,  $Te_0=0$ and  $Te_n=e_{n-1}$ for $n\in\N$, where
$\{e_n\}_{n\in\Z_+}$ is the standard orthonormal basis of $\ell_2$.
Then $\uuu\{aT:a\in\K,\ |a|>1\}$ is a dense $G_\delta$-set.
\end{thmAG}

The following result is of completely different flavor. It is proven
in \cite{muller} for continuous linear operators on Banach spaces
although the proof can be easily adapted \cite{sh1} for continuous
linear operators acting on arbitrary topological vector spaces.

\begin{thmLM} Let $X$ be a complex topological vector
space and $T\in L(X)$. Then $\uu(\F)=\uuu(zT)=\uuu(T)$ for any
$z\in\T$, where $\F=\{wT^n:w\in\T,\ n\in\Z_+\}$. In particular,
$\uuu\{zT:z\in\T\}=\uuu(T)$.
\end{thmLM}

It follows that the family $\{zT:z\in\T\}$ has a common hypercyclic
vector, whenever $T$ is a hypercyclic operator. A result similar to
the above one was recently obtained by Conejero, M\"uller and Peris
\cite{semi} for operators acting on separable $\F$-spaces (see
\cite{sh1} for a proof in a more general setting). Recall that a
family $\{T_t\}_{t\in\R_+}$ of continuous linear operators on a
topological vector space is called an {\it operator semigroup} if
$T_0=I$ and $T_{t+s}=T_tT_s$ for any $t,s\in\R_+$.

\begin{thmCMP} Let $X$ be a topological vector
space and $\{T_t\}_{t\in\R_+}$ be an operator semigroup on $X$.
Assume also that the map $(t,x)\mapsto T_tx$ from $\R_+\times X$ to
$X$ is continuous. Then $\uuu(T_t)=\uu(\F)$ for any $t>0$, where
$\F=\{T_s:s>0\}$. In particular, $\uuu\{T_s:s>0\}=\uuu(T_t)$ for any
$t>0$.
\end{thmCMP}

It follows that if $\{T_t\}_{t\in\R_+}$ is an operator semigroup
such that the map $(t,x)\mapsto T_tx$ is continuous and there exists
$t>0$ for which $T_t$ is hypercyclic, then the family $\{T_s:s>0\}$
has a common hypercyclic vector. Bayart \cite{bay} provided families
of composition operators on the space of holomorphic functions on
$\D$, which have common hypercyclic vectors. Costakis and Sambarino
\cite{cs}, Bayart and Matheron \cite{bm}, Chan and Sanders
\cite{chs} and Gallardo-Guti\'errez and Partington \cite{gp} proved
certain sufficient conditions for a set of families of continuous
linear operators to have a common universal vector. In all the
mentioned papers the criteria were applied to specific sets of
families. For instance, Costakis and Sambarino \cite{cs} proved the
following theorem.

\begin{thmCS} Let $\H(\C)$ be the complex Fr\'echet space of
entire functions on one variable, $D\in L(\H(\C))$ be the
differentiation operator $Df=f'$ and for each $a\in\C$, $T_a\in
L(\H(\C))$ be the translation operator $T_af(z)=f(z-a)$. Then
$\uuu\{T_a:a\in\C^\star\}$, $\uuu\{aT_1:a\in\C^\star\}$ and
$\uuu\{aD:a\in\C^\star\}$ are dense $G_\delta$-sets.
\end{thmCS}

The criteria by Bayart and Matheron were applied to various families
of operators including families of weighted translations on
$L^p(\R)$, composition operators on Hardy spaces $\H^p({\mathbb D})$
and backward weighted shifts on $\ell_p$. We would like to mention
just one example of the application of the criterion from \cite{bm},
which is related to our results.

\begin{exBM} As in Theorem~{\rm CS}, let $T_a$ be translation operators
on $\H(\C)$. For each $s\in\R_+$ and $z\in\T$, consider the family
$\F_{s,z}=\{n^sT_{nz}:n\in\Z_+\}$. Then
$$
\smash{\bigcap_{(s,z)\in\R_+\times \T} \uu(\F_{s,z})\ \ \ \text{is a
dense $G_\delta$-subset of $\H(\C)$.}}\vru
$$
\end{exBM}

Chan and Sanders \cite{cs} found common universal elements of
certain sets of families of backward weighted shifts on $\ell_2$.
Gallardo-Guti\'errez and Partington \cite{gp} proved a modification
of the Costakis--Sambarino criterion and applied it to obtain common
hypercyclic vectors for families of adjoint multipliers and
composition operators on Hardy spaces. Finally, we would like to
mention the following application by Costakis and Mavroudis
\cite{cm} of the Bayart--Matheron criterion.

\begin{thmCM}Let $D$ be the differentiation operator on $\H(\C)$ and
$p$ be a non-constant polynomial. Then $\uuu\{ap(D):a\in\C^\star\}$
is a dense $G_\delta$-set.
\end{thmCM}

Although the most of the mentioned criteria look quite general, they
are basically not applicable to finding common hypercyclic vectors
of families that are not smoothly labeled by {\bf one} real
parameter. Note that although the families in Theorems~AG, CS and CM
are formally speaking labeled by a complex parameter $a$, Theorem~LM
allows to reduce them to families labeled by one real parameter.
Example~BM is, of course, genuinely two-parametric, but it is not
about a common hypercyclic vector. On the other hand, one can
artificially produce huge families of operators with a common
hypercyclic vector. For example, take all operators for which a
given vector is hypercyclic. The following result provides a common
hypercyclic vector for a two-parametric family of operators. It
strengthens the first part of Theorem~CS and kind of improves
Example~BM.

\begin{theorem}\label{t1}Let $T_a$ for $a\in\C$ be the translation
operator $T_af(z)=f(z-a)$ acting on the complex Fr\'echet space
$\H(\C)$ of entire functions on one complex variable. Then
$\uuu\{bT_a:a,b\in\C^\star\}$ is a dense $G_\delta$-set.
\end{theorem}

A common hypercyclic vector from the above theorem is even more
monstrous than the holomorphic monsters provided by Theorem~CS.
Godefroy and Shapiro \cite{gosh} considered adjoint multiplication
operators on function Hilbert spaces. Recall that if $U$ is a
connected open subset of $\C^m$, then a {\it function Hilbert space}
$\H$ on $U$ is a Hilbert space consisting of functions $f:U\to \C$
holomorphic on $U$ such that for any $z\in U$ the evaluation
functional $\chi_z:\H\to\C$, $\chi_z(f)=f(z)$ is continuous. A {\it
multiplier} for $\H$ is a function $\phi:U\to\C$ such that $\phi
f\in \H$ for each $f\in \H$. It is well-known \cite{gosh} that any
multiplier is bounded and holomorphic. Each multiplier gives rise to
the multiplication operator $M_\phi\in L(\H)$, $M_\phi f=\phi f$
(continuity of $M_\phi$ follows from the Banach closed graph
theorem). Its Hilbert space adjoint $M_\phi^\star$ is called an {\it
adjoint multiplication operator}. Godefroy and Shapiro proved that
there is $f\in \H$, which is cyclic for $M_\phi^\star$ for any
non-constant multiplier $\phi$ for $\H$ and  demonstrated that if
$\phi:U\to\C$ is a non-constant multiplier for $\H$ and $\phi(U)\cap
\T\neq\varnothing$, then $M_\phi^\star$ is hypercyclic, see also the
related paper by Bourdon and Shapiro \cite{bosh}. Godefroy and
Shapiro also raised the following question \cite[p.~263]{gosh}.

\begin{qGS} Let $\H$ be a Hilbert function space on a connected open
subset $U$ of $\C^m$. Does the family of all hypercyclic adjoint
multiplications on $\H$ have a common hypercyclic vector?
\end{qGS}

Recall that any $T\in L(\H(\C))$ such that $T$ is not a scalar
multiple of the identity and $TD=DT$ is hypercyclic. The following
question was raised by Richard Aron.

\begin{qA} Let ${\cal D}$ be the family of all continuous linear
operators on $\H(\C)$, which are not scalar multiples of the
identity and which commute with the differentiation operator $D$. Is
it true that there is a common hypercyclic vector for all operators
from the family $\cal D$?
\end{qA}

The next result allows us to answer negatively both of the above
questions.

\begin{theorem}\label{t2} Let $X$ be a complex
topological vector space such that $X^*\neq \{0\}$, $T\in L(X)$ and
$\Lambda$ be a subset of $\R_+\times \C$. Assume also that the
family $\{aT+bI:(a,b)\in\Lambda\}$ has a common hypercyclic vector.
Then the set $\Lambda$ has zero three-dimensional Lebesgue measure.
\end{theorem}

\begin{corollary}\label{c1} The family $\{aD+bI:a>0,\ b\in\C\}$ of
continuous linear operators on $\H(\C)$ does not have a common
hypercyclic vector.
\end{corollary}

\begin{corollary}\label{c2} Let $\H$ be a Hilbert function space on a
connected open subset $U$ of $\C^m$ and $\phi$ be a non-constant
multiplier for $\H$. Then the family
$\{M_{\overline{b}+a\phi}^\star:a>0,\ b\in\C,\
(\overline{b}+a\phi)(U)\cap\T\neq\varnothing\}$ of hypercyclic
operators does not have a common hypercyclic vector.
\end{corollary}

Corollaries~\ref{c1} and~\ref{c2} follow from Theorem~\ref{t2}
because $M_{\overline{b}+a\phi}^\star=aM_\phi^\star+bI$ and the sets
of pairs $(a,b)$ involved in the definition of the families in
Corollaries~\ref{c1} and~\ref{c2} are non-empty open subsets of
$\R_+\times\C$ and therefore have non-zero 3-dimensional Lebesgue
measure. In fact, Theorem~\ref{t2} shows that even relatively small
subfamilies of the families from Questions~GS and~A fail to have
common hypercyclic vectors. As usual, $\H^2(\D)$ is the Hardy space
of the unit disk. It is well-known that $\H^2(\D)$ is a function
Hilbert space on $\D$ and the set of multipliers for $\H^2(\D)$ is
the space $\H^\infty(\D)$ of bounded holomorphic functions
$f:\D\to\C$. Let $\phi\in\H^\infty(\D)$ be non-constant. Using the
mentioned criterion by Godefroy and Shapiro for hypercyclicity of
adjoint multiplications together with the fact that a contraction or
its inverse can not be hypercyclic, we see that
$zM_\phi^\star=M_{\overline{z}\phi}^\star$ is hypercyclic if and
only if $b^{-1}<|z|<a^{-1}$, where $a=\inf\limits_{z\in
\D}|\phi(z)|$ and $b=\sup\limits_{z\in \D}|\phi(z)|$. Probably,
expecting the answer to Question~GS to be negative, Bayart and
Grivaux \cite{bagr} raised the following question.

\begin{qBG} Let $\phi\in \H^\infty(\D)$ be non-constant, $a=\inf\limits_{z\in
\D}|\phi(z)|$ and $b=\sup\limits_{z\in \D}|\phi(z)|$. Is it true
that the family $\{zM_\phi^\star:b^{-1}<|z|<a^{-1}\}$ has common
hypercyclic vectors?
\end{qBG}

We prove a sufficient condition on a family of scalar multiples of a
given operator to have a common hypercyclic vector and use it to
answer Question~BG affirmatively. It is worth noting that
Gallardo-Guti\'errez and Partington \cite{gp} found a partial
affirmative answer to the above question.

\begin{theorem}\label{t3a} Let $X$ be a separable complex $\F$-space,
$T\in L(X)$ and $0\leq a<b\leq\infty$. Assume also that there is a
map $(k,c)\mapsto F_{k,c}$ sending a pair $(k,c)\in\N\times (a,b)$
to a subset $F_{k,c}$ of $X$ satisfying the following properties:
\begin{itemize}\itemsep=-2pt
\item[\rm(\ref{t3a}.1)]$F_{k,c}\subseteq \bigcup\limits_{w\in\T}\ker(T^k-wc^kI)$
for each $(k,c)\in\N\times(a,b);$
\item[\rm(\ref{t3a}.2)]$\{c\in(a,b):F_{k,c}\cap V\neq\varnothing\}$ is open in
$(a,b)$ for any open subset $V$ of $X$ and $k\in\N;$
\item[\rm(\ref{t3a}.3)]$F_c=\bigcup\limits_{k=1}^\infty F_{k,c}$ is dense in $X$
for any $c\in(a,b);$
\item[\rm(\ref{t3a}.4)]For any $k_1,\dots,k_n\in\N$,
there is $k\in\N$ such that $\smash{\bigcup\limits_{j=1}^n
F_{k_j,c}\subseteq F_{k,c}}$ for each $c\in(a,b)$.
\end{itemize}
Then $\uuu\{zT:b^{-1}<|z|<a^{-1}\}$ is a dense $G_\delta$-set.
\end{theorem}

Note that (\ref{t3a}.1) is satisfied if $F_{k,c}\subseteq
\ker(T^k-c^kI)$, which is the case in all following applications of
Theorem~\ref{t3a}. If $X$ is a complex locally convex topological
vector space and $U$ is a non-empty open subset of $\C^m$, then we
say that $f:U\to X$ is {\it holomorphic} if $f$ is continuous and
for each $g\in X^*$, $g\circ f:U\to \C$ is holomorphic.

\begin{theorem}\label{t4} Let $m\in\N$, $X$ be a complex
Fr\'echet space, $T\in L(X)$ and $U$ be a connected open subset of
$\C^m$. Assume also that there exist holomorphic maps $f:U\to X$ and
$\phi:U\to\C$ such that $\phi$ is non-constant, $Tf(z)=\phi(z)f(z)$
for each $z\in U$ and $\spann\{f(z):z\in U\}$ is dense in $X$.
Denote $a=\inf\limits_{z\in U}|\phi(z)|$ and $b=\sup\limits_{z\in
U}|\phi(z)|$. Then $\uuu\{zT:b^{-1}<|z|<a^{-1}\}$ is a dense
$G_\delta$-set.
\end{theorem}

\begin{corollary}\label{co1} Let $m\in\N$, $U$ be connected non-empty
open subset of $\C^m$, $\H$ be a function Hilbert space on $U$ and
$\phi$ be a non-constant multiplier for $\H$, $a=\inf\limits_{z\in
U}|\phi(z)|$ and $b=\sup\limits_{z\in U}|\phi(z)|$. Then
$\uuu\{zT:b^{-1}<|z|<a^{-1}\}$ is a dense $G_\delta$-set.
\end{corollary}

\begin{corollary}\label{co2} Let $T\in L(\H(\C))$ be such that
$TD=DT$ and $T\neq cI$ for any $c\in\C$. Then
$\uuu\{zT:z\in\C^\star\}$ is a dense $G_\delta$-set.
\end{corollary}

\begin{corollary}\label{co3} Let $X$ be a separable Fr\'echet space,
$T\in L(X)$, $0\leq a<b\leq\infty$ and $T\in L(X)$. Assume also that
for any $\alpha,\beta\in\R$ such that $a<\alpha<\beta<b$, there
exists a dense subset $E$ of $X$ and a map $S:E\to E$ such that
$TSx=x$, $\alpha^{-n}T^n x\to 0$ and $\beta^nS^nx\to 0$ for each
$x\in E$. Then $\uuu\{zT:b^{-1}<|z|<a^{-1}\}$ is a dense
$G_\delta$-set.
\end{corollary}

Note that Corollary~\ref{co1} gives an affirmative answer to
Question~BG, Corollary~\ref{co2} contains Theorem~CM as a particular
case, while Corollary~\ref{co3} may be considered as an analog of
the Kitai Criterion. The above results on common hypercyclic vectors
for scalar multiples of a given operator may lead to an impression
that for $0<a<b<\infty$ and a continuous linear operator $T$ on a
Fr\'echet space, hypercyclicity of $aT$ and $bT$ implies the
existence of common hypercyclic vectors for the family $\{cT:a\leq
c\leq b\}$. This impression is utterly false as follows from the
next treorem. For a continuous linear operator $T$ on a topological
vector space $X$, we denote
$$
M_T=\{c>0:cT\ \ \text{is hypercyclic}\}.
$$

\begin{theorem}\label{t5} {\rm I. }\ There exists $S\in L(\ell_2)$
such that $M_S=\{1,2\}$. {\rm II. }\ There exists $T\in L(\ell_2)$
such that $M_T$ is an open interval, but any $A\subset\R_+$ for
which the family $\{cT:c\in A\}$ has common hypercyclic vectors is
of zero Lebesgue measure.
\end{theorem}

\section{Yet another general criterion}

\begin{lemma}\label{gc1} Let $A$ be a set and $X$, $Y$ and $\Omega$ be
topological spaces such that $\Omega$ is compact. For each $a\in A$
let $(\omega,x)\mapsto F_{a,\omega}x$ be a continuous map from
$\Omega\times X$ to $Y$. For any $\omega\in\Omega$ let
$\F_\omega=\{F_{a,\omega}:a\in A\}$ treated as a family of
continuous maps from $X$ to $Y$. Denote
$\uu^*=\bigcap\limits_{\omega\in\Omega}\uu(\F_\omega)$. Then
\begin{equation}\label{uu2}
\smash{G_V=\bigcap_{\omega\in\Omega}\bigcup_{a\in
A}F_{a,\omega}^{-1}(V)\quad \text{is open in $X$ for any open subset
$V$ of $Y$.}}\vru
\end{equation}
Moreover, for any base $\cal V$ of topology of $Y$,
\begin{equation}\label{uu1}
\smash{\uu^*=\bigcap_{V\in{\cal V}}G_V}.\vru
\end{equation}
In particular, $\uu^*$ is a $G_\delta$-set if $Y$ is second
countable.
\end{lemma}

\begin{proof} Let $x\in G_V$. Then for any $\omega\in\Omega$, there
exists $a(\omega)\in A$ such that $F_{a(\omega),\omega}x\in V$.
Continuity of the map $\omega\mapsto F_{a,\omega}x$ implies that for
each $\omega\in\Omega$,
$W_\omega=\{\alpha\in\Omega:F_{a(\omega),\alpha}x\in V\}$ is an open
neighborhood of $\omega$ in $\Omega$. Since any Hausdorff compact
space is regular, for any $\omega\in\Omega$, we can pick an open
neighborhood $W'_\omega$ of $\omega$ in $\Omega$ such that,
$\overline{W'_\omega}\subseteq W_\omega$. Since
$\{W'_\omega:\omega\in\Omega\}$ is an open covering of the compact
space $\Omega$, there are $\omega_1,\dots,\omega_n\in\Omega$ such
that $\Omega=\bigcup\limits_{j=1}^n W'_{\omega_j}$. Continuity of
the map $(\alpha,z)\mapsto F_{a,\alpha}z$ and compactness of
$\overline{W'_\omega}$ imply that for any $j\in\{1,\dots,n\}$, there
is a neighborhood $U_j$ of $x$ in $X$ such that
$F_{a(\omega_j),\alpha}z\in V$ for any $\alpha\in
\overline{W'_{\omega_j}}$ and $z\in U_j$. Let
$U=\bigcap\limits_{j=1}^n U_j$. Since $\Omega=\bigcup\limits_{j=1}^n
W'_{\omega_j}$, for any $z\in U$ and $\omega\in\Omega$, there exists
$j\in\{1,\dots,n\}$ such that $F_{a(\omega_j),\omega}z\in V$. Hence
$U\subseteq G_V$. Thus any point of $G_V$ is interior and therefore
$G_V$ is open. The equality (\ref{uu1}) follows immediately from the
definition of $\uu^*$.
\end{proof}

The main tool in the proof of Theorem~\ref{t1} is the following
criterion. It is a simultaneous generalization of results by Chan
and Sanders \cite[Theorem~2.1]{chs} and Grosse-Erdmann
\cite[Theorem~1]{ge1}. The latter is exactly the next proposition in
the case when $\Omega$ is a singleton.

\begin{proposition}\label{gc} Let $A$ be a set and $X,Y,\Omega$ be
topological spaces such that $X$ is Baire, $Y$ is second countable
and $\Omega$ is compact. For each $a\in A$, let $(\omega,x)\mapsto
F_{a,\omega}x$ be a continuous map from $\Omega\times X$ to $Y$. Let
$\F_\omega=\{F_{a,\omega}:a\in A\}$ for $\omega\in\Omega$ and
$\uu^*=\bigcap\limits_{\omega\in\Omega}\uu(\F_\omega)$. Then $\uu^*$
is a $G_\delta$-subset of $X$. Moreover, the following conditions
are equivalent.
\begin{itemize}\itemsep=-2pt
\item[\rm(\ref{gc}.1)]$\uu^*$ is dense in $X.$
\item[\rm(\ref{gc}.2)]For any non-empty open set $U$ in $X$ and any
non-empty open set $V$ in $Y$, there exists $x\in U$ such that
$V\cap\{F_{a,\omega}x:a\in A\}\neq \varnothing$ for each
$\omega\in\Omega$.
\end{itemize}
\end{proposition}

\begin{proof} Let $\cal V$ be a countable base of the topology of $Y$.
By Lemma~\ref{gc1}, $\uu^*$ is a $G_\delta$-set. Assume that
(\ref{gc}.2) is satisfied. For any $V\in{\cal V}$ and $n\in\N$,
condition (\ref{gc}.2) implies that $G_V$ defined by (\ref{uu2}) is
dense in $X$. By Lemma~\ref{gc1}, each $G_V$ is a dense open subset
of $X$. Since $X$ is Baire, (\ref{uu1}) implies that $\uu^*$ is a
dense $G_\delta$-subset of $X$. Hence (\ref{gc}.2) implies
(\ref{gc}.1). Next, assume that (\ref{gc}.1) is satisfied and $U$,
$V$ are non-empty open subsets of $X$ and $Y$ respectively. Since
$\uu^*$ is dense in $X$, there is $x\in \uu^*\cap U$. Let
$\omega\in\Omega$. Since $x\in\uu(\F_\omega)$, there is $a\in A$
such that $F_{a,\omega}x\in V$. Hence (\ref{gc}.2) is satisfied.
\end{proof}

Using Proposition~\ref{gc} and the fact that in a Baire topological
space the class of dense $G_\delta$ sets is closed under countable
intersections, we immediately obtain the following corollary.

\begin{corollary}\label{gc2} Let $A$ be a set and $X,Y,\Omega$ be
topological spaces such that $X$ is Baire, $Y$ is second countable
and $\Omega$ is the union of its compact subsets $\Omega_n$ for
$n\in\N$. For each $a\in A$, let $(\omega,x)\mapsto F_{a,\omega}x$
be a continuous map from $\Omega\times X$ to $Y$. Let
$\F_\omega=\{F_{a,\omega}:a\in A\}$ for $\omega\in\Omega$ and
$\smash{\uu^*=\bigcap\limits_{\omega\in\Omega}\uu(\F_\omega)}$. Then
$\uu^*$ is a $G_\delta$-subset of $X$. Moreover, the following
conditions are equivalent.
\begin{itemize}\itemsep=-2pt
\item[\rm(\ref{gc2}.1)]$\uu^*$ is dense in $X.$
\item[\rm(\ref{gc2}.2)]For each $n\in\N$, any non-empty open set $U$ in $X$ and any
non-empty open set $V$ in $Y$, there exists $x\in U$ such that
$V\cap\{F_{a,\omega}x:a\in A\}\neq \varnothing$ for each
$\omega\in\Omega_n$.
\end{itemize}
\end{corollary}

Recall that if $X$ is a topological vector space, $A$ is a set and
$\{f_n\}_{n\in\Z_+}$ is a sequence of maps from $A$ to $X$, then we
say that $f_n$ {\it uniformly converges to $0$} on $A$ if for any
neighborhood $W$ of $0$ in $X$, there is $n\in\Z_+$ such that
$f_k(a)\in W$ for any $a\in A$ and any $k\geq n$.

\begin{definition}\label{MO} \rm Let $X$ and $Y$ be topological vector
spaces, $A$ be a set and $\Omega$ be a topological space. We use the
symbol
$$
\L_{\Omega,A}(X,Y)
$$
to denote the set of maps $(\omega,a,n,x)\mapsto T_{\omega,a,n}x$
from $\Omega\times A\times\Z_+\times X$ to $Y$ such that
$T_{\omega,a,n}\in L(X,Y)$ for each $(\omega,a,n)\in \Omega\times
A\times\Z_+$ and the map $(\omega,x)\mapsto T_{\omega,a,n}x$ from
$\Omega\times X$ to $X$ is continuous for any $(a,n)\in
A\times\Z_+$. If $T\in \L_{\Omega,A}(X,Y)$ is fixed,
$\Lambda\subseteq \Z_+$, $u\in X$ and $U$ is a subset of $Y$, we
denote
\begin{equation}\label{MMM}
M(u,\Lambda,U)=\{\omega\in\Omega: \text{$T_{\omega,a,n}u\in U$ for
some $n\in\Lambda$ and $a\in A$}\}.
\end{equation}
\end{definition}

\begin{proposition}\label{gc3} Let $A$ be a set, $X$ be a Baire
topological vector space, $Y$ be a separable metrizable topological
vector space, $\Omega$ be a compact topological space and $T\in
\L_{\Omega,A}(X,Y)$ be such that
\begin{itemize}\itemsep=-2pt
\item[\rm(\ref{gc3}.1)] $E=\{x\in X:T_{\omega,a,n}x\to 0\ \text{as $n\to\infty$
uniformly on $\Omega\times A$}\}$ is dense in $X;$
\item[\rm(\ref{gc3}.2)] for any non-empty open subset $U$ of $Y$,
there exist $m\in\N$ and compact subsets $\Omega_1,\dots,\Omega_m$
of $\Omega$ such that $\Omega=\bigcup\limits_{j=1}^m\Omega_j$ and
for any $j\in\{1,\dots,m\}$, $l\in\Z_+$ and a neighborhood $W$ of
$0$ in $X$, there are a finite set $\Lambda\subset\Z_+$ and $u\in W$
for which $\min\Lambda\geq l$ and $\Omega_j\subseteq
M(u,\Lambda,U)$.
\end{itemize}
Then $\uu^*=\bigcap\limits_{\omega\in\Omega}\uu(\F_\omega)$ is a
dense $G_\delta$-subset of $X$, where
$\F_\omega=\{T_{\omega,a,n}:a\in A,\ n\in\Z_+\}$.
\end{proposition}

\begin{proof} Let $U_0$ be a non-empty open subset of $X$ and $U$ be
a non-empty open subset of $Y$. Pick $y_0\in U$ and a neighborhood
$W$ of zero in $Y$ such that $y_0+W+W\subseteq U$. Then $V=y_0+W$ is
a non-empty open subset of $Y$ and $V+W\subseteq U$. According to
(\ref{gc3}.2), there exist compact subsets $\Omega_1,\dots,\Omega_m$
of $\Omega$ such that $\Omega=\bigcup\limits_{j=1}^m\Omega_j$ and
\begin{equation}\label{gc33}
\begin{array}{l}
\text{for any $j\in\{1,\dots,m\}$, $l\in\Z_+$ a any neighborhood
$W_1$ of $0$ in $X$, there are}\\ \text{a finite set
$\Lambda\subset\Z_+$ and $u\in W_1$ such that $\min\Lambda\geq l$
and $\Omega_j\subseteq M(u,\Lambda,V)$.}
\end{array}
\end{equation}
We shall construct inductively $u_0,\dots,u_{m}\in E\cap U_0$ and
finite sets $\Lambda_1,\dots,\Lambda_{m}\subset \Z_+$ such that for
$0\leq j\leq m$,
\begin{equation}\label{conj}
\text{$\Omega_p\subseteq M(u_j,\Lambda_p,U)$ for $1\leq p\leq j$.}
\end{equation}
By (\ref{gc3}.1), the linear space $E$ is dense in $X$. Hence we can
pick $u_0\in U_0\cap E$, which will serve as the basis of induction.
Assume now that $1\leq q\leq m$ and $u_0,\dots,u_{q-1}\in E\cap U_0$
and finite subsets $\Lambda_1,\dots,\Lambda_{q-1}$ of $\Z_+$
satisfying (\ref{conj}) with $0\leq j\leq q-1$ are already
constructed. We shall construct $u_q\in E\cap U_0$ and a finite
subset $\Lambda_q$ of $\Z_+$ satisfying (\ref{conj}) with $j=q$.
Consider the set
$$
G=\{u\in X:\Omega_p\subseteq M(u,\Lambda_p,U)\ \ \text{for}\ \ 1\leq
p\leq q-1\}.
$$
Since $\Omega_p$ are compact and $U$ is open, Lemma~\ref{gc1}
implies that $G$ is open in $X$. According to (\ref{conj}) with
$j=q-1$, $u_{q-1}\in G$. Since $u_{q-1}\in E$, there exists
$l\in\Z_+$ such that
\begin{equation}\label{sss}
\text{$T_{\omega,a,n}u_{q-1}\in W$ for any $n\geq l$ and any
$(\omega,a)\in\Omega\times A$.}
\end{equation}
Since $u_{q-1}\in G\cap U_0$, and $G\cap U_0$ is open in $X$,
$W_1=(G\cap U_0)-u_{q-1}$ is a neighborhood of $0$ in $X$. According
to (\ref{gc33}), there exists a finite subset $\Lambda_q$ of $\Z_+$
such that
$$
\min\Lambda_q\geq l\ \ \ \text{and}\ \ \ G_1=\{u\in
W_1:\Omega_q\subseteq M(u,\Lambda_q,V)\}\neq\varnothing.
$$
By Lemma~\ref{gc1}, $G_1$ is open in $X$. Since $E$ is dense in $X$,
we can pick $u\in G_1\cap E$. Denote $u_q=u_{q-1}+u$. We shall see
that $u_q$ and $\Lambda_q$ satisfy (\ref{conj}) with $j=q$.

Since $u_{q-1},u\in E$  and $E$ is a linear space, we have $u_q\in
E$. Since $u\in W_1=(G\cap U_0)-u_{q-1}$, we get $u_q\in G\cap U_0$.
In particular, $u_q\in U_0\cap E$ and $u_q\in G$. By definition of
$G$, $\Omega_p\subseteq M(u_q,\Lambda_p,U)$ for $1\leq p\leq q-1$.
Since $u\in G_1$, for any $\omega\in \Omega_q$, there exist
$n_\omega\in \Lambda_q$ and $a_\omega\in A$ such that
$T_{\omega,a_\omega,n_\omega}u\in V$. Since $n_\omega\in \Lambda_q$
and $\min\Lambda_q\geq l$, we have $n_\omega\geq l$. According to
(\ref{sss}), $T_{\omega,a_\omega,n_\omega}u_{q-1}\in W$. The
equality $u_q=u_{q-1}+u$ and linearity of
$T_{\omega,a_\omega,n_\omega}$ imply
$T_{\omega,a_\omega,n_\omega}u_q\in V+W\subseteq U$. Since
$\omega\in \Omega_q$ is arbitrary, $\Omega_q\subseteq
M(u_q,\Lambda_q,U)$. This completes the proof of (\ref{conj}) for
$j=q$ and the inductive construction of $u_0,\dots,u_{m}$ and
$\Lambda_1,\dots,\Lambda_{m}$ satisfying (\ref{conj}).

Since $\Omega$ is the union of $\Omega_j$ with $1\leq j\leq m$,
(\ref{conj}) for $j=m$ implies that $u_m\in U_0$ and
$\Omega=M(u_m,\Z_+,U)$. That is, for any $\omega\in\Omega$ there are
$a\in A$ and $n\in\Z_+$ such that $T_{\omega,a,n}u_m\in U$. Since
$U_0$ and $U$ are arbitrary non-empty open subsets of $X$ and $Y$
respectively, condition (\ref{gc}.2) is satisfied. By
Proposition~\ref{gc}, $\uu^*$ is a dense $G_\delta$-subset of $X$.
\end{proof}

Since for any $\delta>0$, any compact interval of the real line is
the union of finitely many intervals of length $\leq \delta$, we
immediately obtain the following corollary.

\begin{corollary}\label{gc4} Let $A$ be a set, $X$ be a Baire
topological vector space, $Y$ be a separable metrizable topological
vector space, $\Omega$ be a compact interval of $\R$ and $T\in
\L_{\Omega,A}(X,Y)$ be such that $(\ref{gc3}.1)$ is satisfied and
\begin{itemize}\itemsep=-2pt
\item[\rm(\ref{gc4}.2)] for any non-empty open subset $U$ of $Y$,
there exists $\delta>0$ such that for any compact interval
$J\subseteq\Omega$ of length $\leq\delta$, $l\in\Z_+$ and a
neighborhood $W$ of $0$ in $X$, there exist a finite set
$\Lambda\subset\Z_+$ and $u\in W$ for which $\min\Lambda\geq l$ and
$J\subseteq M(u,\Lambda,U)$.
\end{itemize}
Then $\uu^*=\bigcap\limits_{\omega\in\Omega}\uu(\F_\omega)$ is a
dense $G_\delta$-subset of $X$, where
$\F_\omega=\{T_{\omega,a,n}:a\in A,\ n\in\Z_+\}$.
\end{corollary}

\section{Operator groups with the Runge property}

In this section we  prove a statement more general than of
Theorem~\ref{t1}.

\begin{definition}\label{run} \rm Let $X$ be a locally convex
topological vector space and $\{T_z\}_{z\in\C}$ be an operator
group. That is, $T_z\in L(X)$ for each $z\in\C$, $T_0=I$ and
$T_{z+w}=T_zT_w$ for any $z,w\in\C$. We say that $\{T_z\}_{z\in\C}$
has the {\it Runge property} if for any continuous seminorm $p$ on
$X$ there exists $c=c(p)>0$ such that for any finite set $S$ of
complex numbers satisfying $|z-z'|\geq c$ for $z,z'\in S$, $z\neq
z'$, any $\epsilon>0$ and $\{x_z\}_{z\in S}\in X^S$, there is $x\in
X$ such that $p(T_{z}x-x_z)<\epsilon$ for each $z\in S$.
\end{definition}

\begin{lemma}\label{rutr} For each $a\in \C$ let $T_a\in L(\H(\C))$
be the translation operator $Tf(z)=f(z-a)$. Then the group
$\{T_a\}_{a\in\C}$ has the Runge property.
\end{lemma}

\begin{proof} Let $p$ be a continuous seminorm on $\H(\C)$. Then
there exist $a>0$ such that $p(f)\leq q(f)$ for each $f\in \H(\C)$,
where $q(f)=a\max\limits_{|z|\leq a}|f(z)|$. Take any $c>2a$. We
shall show that $c$ satisfies the condition from
Definition~\ref{run}. Let $\epsilon>0$, $S$ be a finite set of
complex numbers such that $|z-z'|\geq c$ for $z,z'\in S$, $z\neq z'$
and $\{f_z\}_{z\in S}\in \H(\C)^S$. For each $z\in S$ consider the
disk $D_z=\{w\in\C:|z+w|\leq a\}$ and let $D=\bigcup\limits_{z\in
S}D_z$. Since $|z-z'|\geq c$ for $z,z'\in S$, $z\neq z'$, the closed
disks $D_z$ are pairwise disjoint. It follows that $\C\setminus D$
is connected. By the classical Runge theorem, any function
holomorphic in a neighborhood of the compact set $D$ can be with any
prescribed accuracy uniformly on $D$ approximated by a polynomial.
Thus there is a polynomial $f$ such that $\sup\limits_{w\in
D_z}|f(w)-f_z(z+w)|<\epsilon/a$ for any $z\in S$. Equivalently,
$\sup\limits_{|w|\leq a}|f(w-z)-f_z(w)|<\delta$ for any $z\in S$.
Using the definitions of $T_z$ and $q$, we obtain $p(T_zf-f_z)\leq
q(T_zf-f_z)<\epsilon$ for each $z\in S$.
\end{proof}

It is also easy to show that the translation group satisfies the
Runge property when acting on the Fr\'echet space $C(\C)$ of
continuous functions $f:\C\to\C$ with the topology of uniform
convergence on compact sets. Recall that an operator semigroup
$\{T_t\}$ is called {\it strongly continuous} if the map
$(t,x)\mapsto T_tx$ is separately continuous.

\begin{theorem}\label{t1a} Let $X$ be a separable Fr\'echet space and
$\{T_z\}_{z\in\C}$ be a strongly continuous operator group on $X$
with the Runge property. Then the family $\{aT_b:a\in\K^\star,\
b\in\C^\star\}$ has a dense $G_\delta$-set of common hypercyclic
vectors.
\end{theorem}

According to Lemma~\ref{rutr}, Theorem~\ref{t1} is a particular case
of Theorem~\ref{t1a}. The rest of this section is devoted to the
proof of Theorem~\ref{t1a}. We need a couple of technical lemmas.

\begin{lemma}\label{tech} For each $\delta,C>0$, there is
$R>0$ such that for any $n\in\N$, there exists a finite set
$S\subset\C$ such that $|z|\in\N$ and $nR+c\leq |z|\leq(n+1)R-c$ for
any $z\in S$, $|z-z'|\geq c$ for any $z,z'\in S$, $z\neq z'$ and for
each $w\in\T$, there exists $z\in S$ such that
$\bigl|w-\frac{z}{|z|}\bigr|<\delta/|z|$.
\end{lemma}

\begin{proof} Without loss of generality, we may assume that
$0<\delta<1$. Pick $m\in\N$ such that $2m\geq c$ and $h\in\N$ such
that $h\geq (40\cdot m)/\delta$. We shall show that $R=hm$ satisfies
the desired condition. Pick $n\in\N$ and consider $k=k(n)\in\N$
defined by the formula $k=\bigl[\frac{\pi (n+1)m}{2\delta
n}\bigr]+1$, where $[t]$ is the integer part of $t\in\R$. For $1\leq
j\leq k$ let $n_j=nR+2jm$. Clearly $n_j$ are natural numbers and
$n_1=nR+2m\geq nR+c$. On the other hand, $n_k=nR+2mk\leq (n+1)R-2m$.
Indeed, the last inequality is equivalent to $2(k+1)\leq h$, which
is an easy consequence of the two inequalities $h>(40\cdot
m)/\delta$ and $k+1\leq \frac{\pi (n+1)m}{2\delta n}+2\leq \frac{\pi
m}{\delta}+2$. Thus,
\begin{equation}\label{nk}
nR+c\leq n_1\leq n_j\leq n_k\leq (n+1)R-2m\leq (n+1)R-c\ \ \text{for
$1\leq j\leq k$}.
\end{equation}
Now we can define a finite set $S$ of complex numbers in the
following way:
\begin{equation}\label{SS}
S=\{z_{j,l}:1\leq j\leq k,\ 0\leq l\leq 2nh-1\},\ \ \text{where}\ \
z_{j,l}=n_j\exp\Bigl(\frac{\pi i(lk+j)}{nhk}\Bigr)
\end{equation}
and $\exp(z)$ stands for $e^z$. Clearly for each $z_{j,l}\in S$, we
have $|z_{j,l}|=n_j\in\N$. Moreover, according to (\ref{nk}),
$nR+c\leq |z|\leq (n+1)R-c$ for any $z\in S$. Next, let $z,z'\in S$
and $z\neq z'$. Then $z=z_{j,l}$ and $z'=z_{p,q}$ for $1\leq j,p\leq
k$, $0\leq l,q\leq 2nh-1$ and $(j,l)\neq (p,q)$. If $j\neq p$, then
$|z-z'|\geq ||z|-|z'||=|n_j-n_p|=2m|j-p|\geq2m\geq c$. If $j=p$,
then $l\neq q$ and
$$
|z-z'|=n_j\Bigl|\exp\Bigl(\frac{\pi i
l}{nh}\Bigr)-\exp\Bigl(\frac{\pi i q}{nh}\Bigr)\Bigr|\geq
n_j\Bigl|\exp\Bigl(\frac{\pi
i}{nh}\Bigr)-1\Bigr|=2n_j\sin\Bigl(\frac{\pi}{2nh}\Bigr).
$$
The inequality $\sin x\geq \frac{2x}{\pi}$ for $0\leq x\leq \pi/2$,
the inequality $n_j>nR$ and the equality $R=hm$ imply $|z-z'|\geq
\frac{4\pi n_j}{2\pi nh}=\frac{2n_j}{nh}>\frac{2nR}{nh}=2m\geq c$.
Thus $|z-z'|\geq c$ for any $z,z'\in S$, $z\neq z'$. Finally,
consider the set $\Sigma=\{z/|z|:z\in S\}$. Clearly
\begin{equation*}
\Sigma=\Bigl\{\exp\Bigl(\frac{\pi i(lk+j)}{nhk}\Bigr):{{1\leq j\leq
k,\atop 0\leq l\leq 2nh-1}}\Bigr\}= \Bigl\{\exp\Bigl(\frac{\pi
ij}{nhk}\Bigr):1\leq j\leq 2nhk\Bigr\}=\{z\in\C:z^{2nhk}=1\}.
\end{equation*}
It immediately follows that
$$
\sup_{w\in\T}\min_{z\in\Sigma}|w-z|=\Bigl|1-\exp\Bigl(\frac{\pi
i}{2nhk}\Bigr)\Bigr|=2\sin\Bigl(\frac{\pi}{4nhk}\Bigr)\leq
\frac{\pi}{2nhk}=\frac{\pi m}{2nRk}.
$$
Since $k>\frac{\pi (n+1)m}{2\delta n}$, we get
$\sup\limits_{w\in\T}\min\limits_{z\in\Sigma}|w-z|<\delta
(n+1)^{-1}R^{-1}$. That is, for any $w\in\T$, there exists $z\in S$
such that $\bigl|w-\frac{z}{|z|}\bigr|<\frac{\delta}{R(n+1)}$. Since
$|z|<R(n+1)$, we obtain $\bigl|w-\frac{z}{|z|}\bigr|<\delta/|z|$,
which completes the proof.
\end{proof}

\begin{lemma}\label{lll} Let $X$ be a locally convex
topological vector space and $\{T_z\}_{z\in\C}$ be an operator group
on $X$ such that the map $(u,h)\mapsto T_hu$ from $X\times\C$ to $X$
is continuous. Let also $x\in X$ and $p$ be a continuous seminorm on
$X$. Then there exist a continuous seminorm $q$ on $X$ and
$\delta>0$ such that $p\leq q$ and for any $a\in\R$, $w\in\T$,
$n\in\N$ and $y\in X$ satisfying $q(x-e^{an}T_{wn}y)<1$, we have
$p(x-e^{bn}T_{zn}y)<1$ whenever $b\in\R$ and $z\in\T$ are such that
$|a-b|<\delta/n$ and $|w-z|<\delta/n$.
\end{lemma}

\begin{proof} Since the map $(u,h)\mapsto T_hu$ from $X\times\C$ to $X$ is
continuous, there is $\theta>0$ and a continuous seminorm $q$ on $X$
such that $p(x-T_hx)\leq 1/4$ and $p(T_hu)\leq q(u)/4$ for any $u\in
X$ whenever $|h|\leq \theta$. In particular, $p(u)\leq q(u)/4\leq
q(u)$ for each $u\in X$. Pick $r\in(0,\theta)$ and assume that
$a,b\in\R$, $w,z\in\T$, $n\in\N$ and $y\in X$ are such that
$q(x-e^{an}T_{wn}y)<1$, $|a-b|<r/n$ and $|w-z|<r/n$. Then
$p(e^{an}T_{wn}y)\leq q(e^{an}T_{wn}y)\leq q(x)+1$. Since
$|a-b|<r/n$, we have $|e^{(b-a)n}-1|<e^r-1$. Hence
\begin{equation}\label{ee1}
p(e^{bn}T_{wn}y-e^{an}T_{wn}y) =|e^{(b-a)n}-1|p(e^{an}T_{wn}y)\leq
(e^r-1)(q(x)+1).
\end{equation}
Since $|nw-nz|<r<\theta$ and $p(T_hu)\leq q(u)/4$ for any $u\in X$
whenever $|h|\leq \theta$, we have
$$
p(T_{(z-w)n}x-e^{an}T_{zn}y)=p(T_{(z-w)n}(x-e^{an}T_{wn}y))\leq
q(x-e^{an}T_{wn}y)/4<1/4.
$$
Since $|(z-w)n|<r<\theta$, we get $p(x-T_{(z-w)n}x)\leq 1/4$. Using
this inequality together with the last display and the triangle
inequality, we obtain $p(x-e^{an}T_{zn}y)\leq 1/2$. The latter
together with (\ref{ee1}) and the triangle inequality gives
$p(x-e^{bn}T_{zn}y)<(e^r-1)(q(x)+1)+1/2$. Hence any
$\delta\in(0,\theta)$ satisfying $(e^\delta-1)(q(x)+1)<1/2$,
satisfies also the desired condition.
\end{proof}

\subsection{Proof of Theorem~\ref{t1a}}

By Theorems~LM and~CMP, $\uuu(bT_a)=\uuu(b'T_{a'})$ if $|b|=|b'|$
and $a/a'\in\R_+$. Hence the set of common hypercyclic vectors of
the family $\{aT_b:a\in\K^\star,\ b\in\C^\star\}$ coincides with the
set $G$  of common hypercyclic vectors for the family
$\{e^bT_a:(a,b)\in\T\times \R\}$. Thus it remains to show that $G$
is a dense $G_\delta$-subset of $X$. Fix $d>0$. According to
Corollary~\ref{gc2}, it suffices to demonstrate that
\begin{equation}\label{W}
\begin{array}{l}
\text{for any non-empty open subsets $U$ and $V$ of $X$, there is
$y\in U$ such that}\\ \text{for any $a\in\T$ and $b\in[-d,d]$ there
is $n\in\N$ for which $e^{bn}T_{an}y\in V$.}
\end{array}
\end{equation}

Pick a continuous seminorm $p$ on $X$ and $u,x\in X$ such that
$\{y\in X:p(u-y)<1\}\subseteq U$ and $\{y\in X:p(x-y)<1\}\subseteq
V$. By the uniform boundedness principle \cite{sch}, strong
continuity of $\{T_z\}_{z\in\C}$ implies that the map $(z,v)\mapsto
T_zv$ from $\C\times X$ to $X$ is continuous. By Lemma~\ref{lll},
there is a continuous seminorm $q$ on $X$ and $\delta>0$ such that
$p(v)\leq q(v)$ for any $v\in X$ and
\begin{equation}\label{L}
\begin{array}{l}
\text{for any $a,b\in\R$, $w,z\in\T$, $n\in\N$ and $y\in X$
satisfying $q(x-e^{an}T_{wn}y)<1$,}
\\
\text{$|a-b|<\delta/n$ and $|w-z|<\delta/n$, we have
$p(g-e^{bn}T_{zn}y)<1$.}
\end{array}
\end{equation}
Since $\{T_z\}_{z\in\C}$ has the Runge property, there is $c>0$ such
that
\begin{equation}\label{J}
\begin{array}{l}
\text{for any finite set $S\subset\C$ with $|z-z'|\geq c$ for
$z,z'\in S$, $z\neq z'$, any $\epsilon>0$ and any}
\\
\text{$\{x_z\}_{z\in S}\in X^S$, there exists $y\in X$ such that
$q(T_{z}y-x_z)<\epsilon$ for any $z\in S$.}
\end{array}
\end{equation}

Let $R>0$ be the number provided by Lemma~\ref{tech} for the just
chosen $\delta$ and $c$. By Lemma~\ref{tech}, for each $n\in\N$
there is a finite set $S_n\subset\C$ such that $|z|\in\N$ and
$nR+c\leq|z|\leq(n+1)R-c$ for any $z\in S_n$, $|z-z'|\geq c$ for any
$z,z'\in S_n$, $z\neq z'$ and for each $w\in\T$, there is $z\in S_n$
such that $\bigl|w-\frac{z}{|z|}\bigr|<\frac{\delta}{|z|}$. Since
$\sum\limits_{n=1}^\infty n^{-1}=\infty$, we can pick
$d_1,\dots,d_k\in[-d,d]$ for which
\begin{equation}\label{cover}
\smash{[-d,d]\subseteq \bigcup_{n=1}^k \Bigl(d_n-\frac{\delta
R^{-1}}{n+1},d_n+\frac{\delta R^{-1}}{n+1}\Bigr).}
\end{equation}
Let $S=\bigcup\limits_{n=1}^kS_n$ and $\Lambda=S\cup\{0\}$. It is
straightforward to see that $\Lambda$ is a finite set, $|z|\in\Z_+$
for any $z\in\Lambda$ and $|z-u|\geq c$ for any $z,u\in \Lambda$,
$z\neq u$. Let $N=\max\{|z|:z\in\Lambda\}$ and $\epsilon=d^{-N}$. By
(\ref{J}), there is $y\in X$ such that $q(u-y)<\epsilon$ and
$q(T_zy-e^{-c_n|z|}x)<\epsilon$ for each $z\in S$. Then $p(u-y)\leq
q(u-y)<\epsilon<1$ and therefore $f\in U$. By definition of
$\epsilon$, $q(x-e^{c_n|z|}T_zy)<1$ for each $z\in S$. Let now
$a\in\T$ and $b\in [-d,d]$. According to (\ref{cover}), there is
$n\in\{1,\dots,k\}$ such that $|b-d_n|<\frac{\delta R^{-1}}{n+1}$.
By the mentioned property of the set $S_n$, we can choose $z\in S_n$
such that $\bigl|a-\frac{z}{|z|}\bigr|<\frac{\delta}{|z|}$. Since
$|z|<R(n+1)$, we have $|b-d_n|<\frac{\delta}{|z|}$. By (\ref{L}),
$p(x-e^{b|z|}T_{a|z|}y)<1$. Hence $e^{b|z|}T_{a|z|}f\in V$, which
completes the proof of (\ref{W}) and that of Theorem~\ref{t1a}.

\section{Scalar multiples of a fixed operator}

In this section we shall prove Theorems~\ref{t3a} and~\ref{t4} as
well as Corollaries~\ref{co1}, \ref{co2} and~\ref{co3}. Recall that
a subset $A$ of a vector space is called {\it balanced} if $zx\in A$
for any $x\in A$ and $z\in\K$ satisfying $|z|\leq 1$. It is
well-known that any topological vector space has a base of open
neighborhoods of zero consisting of balanced sets. For two subsets
$A,B$ of a vector space $X$ we say that $A$ {\it absorbs} $B$ if
there exists $c>0$ such that $B\subseteq zA$ for any $z\in\K$
satisfying $|z|\geq c$. Obviously, if $A$ is balanced, then $A$
absorbs $B$ if and only if there is $c>0$ for which $B\subseteq cA$.

\begin{lemma}\label{sm1} Let $X$ be a topological vector space and
$U$ be a non-empty open subset of $X$. Then there exists a non-empty
open subset $V$ of $X$ and a balanced neighborhood $W$ of zero in
$X$ such that $V+W\subseteq U$ and $W$ absorbs $V$.
\end{lemma}

\begin{proof} Pick $u\in U$ and a balanced neighborhood $W_0$ of
zero in $X$ such that $u+W_0+W_0+W_0\subseteq U$. Denote $V=u+W_0$
and $W=W_0+W_0$. Clearly $V$ is a non-empty open subset of $X$, $W$
is a balanced neighborhood of 0 in $X$ and
$V+W=u+W_0+W_0+W_0\subseteq U$. Since $W_0$ is a neighborhood of 0
in $X$, we can pick $c\geq 1$ such that $u\in cW_0$. Since $W_0$ is
balanced and $c\geq 1$, $W_0\subseteq cW_0$ and therefore
$V=u+W_0\subseteq cW_0+W_0\subseteq c(W_0+W_0)=cW$. Since $W$ is
balanced, $W$ absorbs $V$.
\end{proof}

To any continuous linear operator $T$ on a complex topological
vector space $X$ there corresponds ${\bf T}\in\L_{\R,\T}(X,X)$
defined by the formula ${\bf T}_{t,w,n}x=we^{tn}T^nx$. We will use
the symbol $M(T,u,\Lambda,U)$ to denote the sets defined in
(\ref{MMM}) for ${\bf T}$. In other words, for
$\Lambda\subseteq\Z_+$, $t\in\R$, $u\in X$ and a subset $U$ of $X$,
we write
\begin{equation*}
M(T,u,\Lambda,U)=\{t\in\R: \text{$we^{tn}T^nu\in U$
 for some $n\in\Lambda$ and $w\in\T$}\}.
\end{equation*}

\begin{lemma}\label{sm2} Let $X$ be a complex topological vector space,
$W$ be a balanced neighborhood of $0$ in $X$, $c>0$, $k\in\N$ and
$\delta\in(0,(2ck)^{-1}]$. Then for any $m\in\N$, any
$\alpha\in[-c,c]$, any $w\in\T$, any neighborhood $W_0$ of zero in
$X$ and any $x\in cW$ such that $T^kx=we^{-\alpha k}x$, there exist
$u\in W_0$ and a finite set $\Lambda\subset\N$ such that
$\min\Lambda\geq m$ and $[\alpha+\delta,\alpha+2\delta]\subseteq
M(T,u,\Lambda,x+W)$.
\end{lemma}

\begin{proof} Let $\alpha\in[-c,c]$, $w\in\T$ and any $x\in
cW$ be such that $T^kx=we^{-\alpha k}x$. For each $p\in\N$ consider
$u_p=e^{-2\delta kp}x$. Since $T^kx=we^{-\alpha k}x$, we see that
for $0\leq j\leq p$,
$$
T^{(p+j)k}u_p=e^{-\alpha(p+j)k}e^{-2\delta
kp}w^{p+j}x=\exp\Bigl(-(p+j)k\Bigl(\alpha+\frac{2p\delta}{p+j}\Bigr)\Bigr)w^{p+j}x.
$$
That is,
\begin{equation}\label{thj}
w_je^{(p+j)k\theta_j}T^{(p+j)k}u_p=x\ \ \text{for $1\leq j\leq p$,
where $\displaystyle\theta_j=\alpha+\frac{2\delta p}{p+j}$ and
$w_j=w^{-p-j}\in\T$.}
\end{equation}
Let now $0\leq l\leq p-1$ and $\theta\in[\theta_{l+1},\theta_l]$.
Since
$e^{(p+l)k\theta}T^{(p+l)k}u_p=e^{(p+l)k(\theta-\theta_l)}e^{(p+l)k\theta_l}T^{(p+l)k}u_p$,
using (\ref{thj}) with $j=l$, we obtain
$$
w_le^{(p+l)k\theta}T^{(p+l)k}u_p=e^{(p+l)k(\theta-\theta_l)}x=x+(e^{(p+l)k(\theta-\theta_l)}-1)x.
$$
Taking into account that $-(\theta_l-\theta_{l+1})\leq
\theta-\theta_l\leq 0$ and using the inequality $0\leq 1-e^{-t}\leq
t$ for $t\geq 0$, we see that $|e^{(p+l)k(\theta-\theta_l)}-1|\leq
(p+l)k(\theta_l-\theta_{l+1})$. This inequality, the inclusion $x\in
cW$ the last display and the fact that $W$ is balanced imply that
$$
w_le^{(p+l)k\theta}T^{(p+l)k}u_p\in
x+c|e^{(p+l)k(\theta-\theta_l)}-1|W\subseteq
x+c(p+l)k(\theta_l-\theta_{l+1})W.
$$
Since  $\theta_l-\theta_{l+1}=\frac{2p\delta}{(p+l)(p+l+1)}\leq
\frac{2\delta}{p+l}$ and $\delta\leq (2ck)^{-1}$, we have
$c(p+l)k(\theta_l-\theta_{l+1})\leq 1$. Thus according to the above
display, $w_le^{(p+l)k\theta}T^{(p+l)k}u_p\in x+W$ whenever
$\theta\in [\theta_{l+1},\theta_l]$. It follows that
$[\theta_{l+1},\theta_l]\subseteq M(T,u_p,\Lambda_p,x+W)$ for $0\leq
l\leq p-1$, where $\Lambda_p=\{(p+j)k:0\leq j\leq p\}$. Since the
sequence $\{\theta_j\}_{0\leq j\leq p}$ decreases,
$\theta_0=\alpha+2\delta$ and $\theta_p=\alpha+\delta$, we see that
$[\alpha+\delta,\alpha+2\delta]=\bigcup\limits_{l=0}^{p-1}[\theta_{l+1},\theta_l]$.
Since $[\theta_{l+1},\theta_l]\subseteq M(T,u_p,\Lambda_p,x+W)$ for
$0\leq l\leq p-1$, we have $[\alpha+\delta,\alpha+2\delta]\subseteq
M(T,u_p,\Lambda_p,x+W)$ for any $p\in\N$. Clearly
$\min\Lambda_p=pk\to\infty$ and $u_p=e^{-2\delta kp}x\to 0$ in $X$
as $p\to\infty$. Thus we can pick $p\in\N$ such that
$\min\Lambda_p>m$ and $u_p\in W_0$. Then $u=u_p$ and
$\Lambda=\Lambda_p$ for such a $p$ satisfy all desired conditions.
\end{proof}

We shall prove a statement more general than Theorem~\ref{t3a}.

\begin{theorem}\label{t3} Let $X$ be a separable
complex $\F$-space, $T\in L(X)$ and $0\leq a<b\leq\infty$. Assume
also that the following condition is satisfied.
\begin{itemize}\itemsep=-2pt
\item[\rm(\ref{t3}.1)]For any compact interval $J\subset(a,b)$ and
any non-empty open subset $V$ of $X$, there exists $k=k(J,V)\in\N$
and a dense subset $C=C(J,V)$ of $J$ such that
$$
\smash{V\cap \bigcup_{w\in\T}\ker(T^k-wc^kI)\neq\varnothing\ \
\text{for each $c\in C$}.}
$$
\end{itemize}
Then $\uuu\{zT:b^{-1}<|z|<a^{-1}\}$ is a dense $G_\delta$-set.
\end{theorem}

\begin{proof} Let $\alpha_0,\alpha,\beta\in\R$ be such that
$b^{-1}<e^{\alpha_0}<e^\alpha<e^{\beta}<a^{-1}$. For each $\omega\in
[\alpha,\beta]$ consider the family ${\cal F}_\omega=\{ze^{\omega
n}T^n:z\in\T,\ n\in\Z_+\}$. We shall apply Corollary~\ref{gc4} with
$A=\T$, $T_{\omega,a,n}=ae^{\omega n}T^n$ and
$\Omega=[\alpha,\beta]$. First, pick a compact interval $J\subset
(a,e^{-\beta})$. For each non-empty open subset $V_0$ of $X$, we can
use (\ref{t3}.1) to find $x\in V_0$, $k\in\N$, $r\in J$ and $w\in\T$
such that $T^{k}x=wr^{k}x$. The latter equality implies that $x$ is
a sum of finitely many eigenvectors of $T$ corresponding to
eigenvalues $\lambda_j$ with $|\lambda_j|=r<e^{-\beta}$. Hence
$e^{\beta n}T^nx\to 0$ as $n\to\infty$. Since $V_0$ is an arbitrary
non-empty open subset of $X$ and $x\in V_0$, we see that the space
$E=\{x\in X:e^{\beta n}T^nx\to 0\}$ is dense in $X$. It immediately
follows that
$$
\text{for any $x\in E$, $ze^{\omega n}T^nx\to 0$ as $n\to\infty$
uniformly for $(z,\omega)\in\T\times[\alpha,\beta]$}.
$$
Hence (\ref{gc3}.1) is satisfied. Let now $U$ be a non-empty open
subset of $X$. By Lemma~\ref{sm1}, there exists a balanced
neighborhood $W$ of zero in $X$ and a non-empty open subset $V$ of
$X$ such that $V+W\subseteq U$ and $W$ absorbs $V$. Since $W$
absorbs $V$, there is $c>0$ such that $V\subseteq cW$. According to
(\ref{t3}.2), we can pick $k\in\N$ and a dense subset $R$ of
$[\alpha_0,\beta]$ for which
\begin{equation}\label{qbxc}
V\cap\bigcup_{w\in\T}\ker(T^k-we^{-rk}I)\neq\varnothing\ \ \text{for
any $r\in R$}.
\end{equation}
Let $\delta_0=\min\{(2ck)^{-1},\alpha-\alpha_0\}$ and $r\in R$. By
(\ref{qbxc}), we can pick $w_r\in\T$ and $x_r\in V\subseteq cW$ such
that $T^kx_r=wr^{-rk}x_r$. By Lemma~\ref{sm2}, for any neighborhood
$W_0$ of zero in $X$ and any $m\in\N$, there exist $u\in W_0$ and a
finite set $\Lambda\subset\N$ satisfying $\min\Lambda\geq m$ and
$[r+\delta_0,r+2\delta_0]\subseteq M(T,u,\Lambda,x_r+W)$. Pick
$\delta\in(0,\delta_0)$. Since $R$ is dense in $[\alpha_0,\beta]$
and $\delta_0\leq \alpha-\alpha_0$, it is easy to see that each
compact interval $J\subseteq[\alpha,\beta]$ of length at most
$\delta$ is contained in $[r+\delta_0,r+2\delta_0]$ for some $r\in
R$. Thus for each compact interval $J\subseteq[\alpha,\beta]$ of
length at most $\delta$, any neighborhood $W_0$ of zero in $X$ and
any $m\in\N$, there exist $r\in R$, $u\in W_0$ and a finite set
$\Lambda$ such that $\min\Lambda\geq m$ and $J\subseteq
M(T,u,\Lambda,x_r+W)$. The latter inclusion means that for each
$t\in J$, there exist $w_t\in\T$ and $n_t\in\Lambda$ such that
$w_tT^{n_t}u\in x_r+W$. Since $x_r\in V$ and $V+W\subseteq U$, we
get $w_tT^{n_t}u\in U$. That is, for any compact interval
$J\subseteq[\alpha,\beta]$ of length at most $\delta$, any
neighborhood $W_0$ of zero in $X$ and any $m\in\N$, there exist
$u\in W_0$ and a finite set $\Lambda$ such that $\min\Lambda\geq m$
and $J\subseteq M(T,u,\Lambda,U)$. Thus (\ref{gc4}.2)  is also
satisfied. By Corollary~\ref{gc4},
$$
H_{\alpha,\beta}=\bigcap_{\omega\in[\alpha,\beta]}\uu(\F_\omega)\ \
\text{is a dense $G_\delta$-subset of $X$ whenever
$b^{-1}<e^{\alpha}<e^{\beta}<a^{-1}$.}
$$
By Theorem~LM, $\uu(\F_\omega)=\uuu(ze^\omega T)$ for any
$\omega\in\R$ and $z\in\T$. Hence
$H_{\alpha,\beta}=\uuu\{zT:e^\alpha\leq|z|\leq e^{\beta}\}$. From
the above display it now follows that $\uuu\{zT:b^{-1}<|z|<a^{-1}\}$
is a dense $G_\delta$-subset of $X$ as the intersection of a
countable family of dense $G_\delta$-sets.
\end{proof}

\subsection{Proof of Theorem~\ref{t3a}}

We shall prove Theorem~\ref{t3a} by means of applying
Theorem~\ref{t3}. To do this it suffices to demonstrate that
(\ref{t3}.1) is satisfied. Let $J\subset(a,b)$ be a compact interval
and $V$ be a non-empty open subset of $X$. For any $k\in\N$ let
$O_k=\{c\in(a,b):F_{k,c}\cap V\neq\varnothing\}$.  By (\ref{t3a}.2),
$O_k$ are open subsets of $(a,b)$. According to (\ref{t3a}.3),
$\{O_k:k\in\N\}$ is an open covering of $(a,b)$. Since $J$ is
compact, we can pick $k_1,\dots,k_n\in\N$ such that
$J\subseteq\bigcup\limits_{j=1}^n O_{k_j}$. By (\ref{t3a}.4), there
is $k\in\N$ for which $\bigcup\limits_{j=1}^n F_{k_j,c}\subseteq
F_{k,c}$ for any $c\in(a,b)$. Hence $O_k\supseteq
\bigcup\limits_{j=1}^n O_{k_j}\supseteq J$. It follows that for any
$c\in J$, there is $x\in F_{k,c}\cap V$. According to (\ref{t3a}.1),
there is $w\in\T$ for which $x\in\ker(T^k-wc^kI)$. Thus $V\cap
\bigcup\limits_{w\in\T}\ker(T^k-wc^kI)\neq\varnothing$ for any $c\in
J$. That is, (\ref{t3}.1) is satisfied with $C=J$. It remains to
apply Theorem~\ref{t3} to conclude the proof of Theorem~\ref{t3a}.

\subsection{Proof of Theorem~\ref{t4}}

Recall that a map $h$ from a topological space $X$ to a topological
space $Y$ is called {\it open} if $h(U)$ is open in $Y$ for any open
subset $U$ of $X$. Recall also that a subset $A$ of a connected open
subset $U$ of $\C^m$ is called a {\it set of uniqueness} if any
holomorphic function $\phi:U\to\C$ vanishing on $A$ is identically
zero. The following lemma contains few classical results that can be
found in virtually any book on complex analysis.

\begin{lemma}\label{com1} Let $m\in\N$ and $U$ be a connected open
subset of $\C^m$. Then any non-empty open subset of $U$ is a set of
uniqueness and any non-constant holomorphic map $\phi:U\to\C$ is
open. Moreover, if $m=1$, then any subset of $U$ with at least one
limit point in $U$ is a set of uniqueness.
\end{lemma}

We need the following generalization of the last statement of
Lemma~\ref{com1} to the case $m>1$. Although it is probably known,
the author was unable to locate a reference.

\begin{lemma}\label{com2} Let $m\in\N$, $U$ be a connected open
subset of $\C^m$, $\phi:U\to\C$ be a non-constant holomorphic map
and $A$ be a subset of $\C$ with at least one limit point in
$\phi(U)$. Then $\phi^{-1}(A)$ is a set of uniqueness. In
particular, if $a=\inf\limits_{z\in U}|\phi(z)|$,
$b=\sup\limits_{z\in U}|\phi(z)|$, $c\in(a,b)$ and $G$ is a dense
subset of $\T$, then $\phi^{-1}(cG)$ is a set of uniqueness.
\end{lemma}

\begin{proof} Assume the contrary. Then there exists a non-zero
holomorphic function $f:U\to\C$ such that
$f\bigr|_{\phi^{-1}(A)}=0$. Let $a\in \phi(U)$ be a limit point of
$A$ and $w\in U$ be such that $\phi(w)=a$. Pick a convex open subset
$V$ of $\C^m$ such that $w\in V\subseteq U$. For any complex
one-dimensional linear subspace $L$ of $\C^m$, $V_L=(w+L)\cap V$ can
be treated as a convex open subset of $\C$. If
$\phi_L=\phi\bigr|_{V_L}$ is non-constant, then by Lemma~\ref{com1},
$\phi_L:V_L\to \C$ is open. Since $a=\phi(w)$ is a limit point of
$A$, it follows that $w$ is a limit point of $\phi_L^{-1}(A)$. Using
the one-dimensional uniqueness theorem, we see that $\phi_L^{-1}(A)$
is a set of uniqueness in $V_L$. Since $f$ vanishes on
$\phi^{-1}(A)\supseteq \phi_L^{-1}(A)$, $f\bigr|_{V_L}=0$. On the
other hand, if $\phi_L$ is constant, then $(\phi-a)\bigr|_{V_L}=0$.
Since $L$ is arbitrary, we have $f(\phi-a)\bigr|_V=0$. Since $V$,
being a non-empty open subset of $U$, is a set of uniqueness, we
have $f\cdot(\phi-a)=0$. Since $f\not\equiv 0$, there is a non-empty
open subset $W$ of $U$ such that $f(z)\neq 0$ for any $z\in W$. The
equality $f\cdot(\phi-a)=0$ implies that $\phi(z)=a$ for any $z\in
W$. Since $W$ is a set of uniqueness, $\phi\equiv a$. We have
arrived to a contradiction. Thus $\phi^{-1}(A)$ is a set of
uniqueness.

Assume now that $a=\inf\limits_{z\in U}|\phi(z)|$,
$b=\sup\limits_{z\in U}|\phi(z)|$, $c\in(a,b)$ and $G$ is a dense
subset of $\T$. Since $U$ is connected $c\T\cap \phi(U)\neq
\varnothing$. Since $\phi$ is open, the set $\phi(U)$ is open in
$\C$. Thus density of $G$ in $\T$ implies that $cG\cap \phi(U)$ is
dense in $c\T\cap \phi(U)$, which is an open subset of $c\T$. Hence
$cG$ has plenty of limit points in $\phi(U)$ and it remains to apply
the first part of the lemma.
\end{proof}

We shall prove Theorem~\ref{t4} by means of applying
Theorem~\ref{t3a}. First, note that density of $\spann\{f(z):z\in
U\}$ implies separability of $X$. Let
$$
F_{k,c}=\spann\{f(z):z\in U,\ \phi(z)^k=c^k\}\ \ \text{for $k\in\N$
and $c\in(a,b)$.}
$$
In order to apply Theorem~\ref{t3a} it suffices to verify that the
map $(k,c)\mapsto F_{k,c}$ satisfies conditions
(\ref{t3a}.1--\ref{t3a}.4). First, from the equality
$Tf(z)=\phi(z)f(z)$ it follows that $T^kx=c^kx$ for any $x\in
F_{k,c}$. Hence (\ref{t3a}.1) is satisfied. Clearly
$F_{k,c}\subseteq F_{m,c}$ whenever $k$ is a divisor of $m$. Hence
for any $c\in (a,b)$ and any $k_1,\dots,k_n\in\N$,
$F_{k_j,c}\subseteq F_{k,c}$ for $1\leq j\leq n$, where
$k=k_1\cdot{\dots}\cdot k_n$. Thus (\ref{t3a}.4) is satisfied. It is
easy to see that
$$
F_c=\bigcup_{k=1}^\infty F_{k,c}=\spann\{f(z):\phi(z)\in c\G\},\ \
\text{where}\ \ \G=\{z\in\T:z^k=1\ \text{for some}\ k\in\N\}.
$$
In order to prove (\ref{t3a}.3), we have to show that $F_c$ is dense
in $X$. Assume the contrary. Since $F_c$ is a vector space and $X$
is locally convex, we can pick $g\in X^*$  such that $g\neq 0$ and
$g(x)=0$ for each $x\in F_c$. In particular, $g(f(z))=0$ whenever
$\phi(z)\in c\G$. By Lemma~\ref{com2}, $\phi^{-1}(c\G)$ is a set of
uniqueness. Since the holomorphic function $g\circ f$ vanishes on
$\phi^{-1}(c\G)$, it is identically zero. Hence $g(f(z))=0$ for any
$z\in U$, which contradicts the density of $\spann\{f(z):z\in U\}$
in $X$. This contradiction completes the proof of (\ref{t3a}.3). It
remains to verify (\ref{t3a}.2). Let $k\in\N$, $V$ be a non-empty
open subset of $X$ and $G=\{c\in(a,b):F_{k,c}\cap
V\neq\varnothing\}$. We have to show that $G$ is open in $\R$. Let
$c\in G$. Then there exist $z_1,\dots,z_n\in U$ and
$\lambda_1,\dots,\lambda_n\in\C$ such that $\phi(z_j)^k=c^k$ for
$1\leq j\leq n$ and $\sum\limits_{j=1}^n\lambda_jf(z_j)\in V$. Since
$f$ is continuous, we can pick $\epsilon>0$ such that
$z_j+\epsilon\D^m\subset U$ for $1\leq j\leq n$ and
$\sum\limits_{j=1}^n\lambda_jf(w_j)\in V$ for any choice of $w_j\in
z_j+\epsilon\D^m$. By Lemma~\ref{com1}, $\phi$ is open and therefore
there exists $\delta>0$ such that $\phi(z_j)+c\delta\D\subseteq
\phi(z_j+\epsilon\D^m)$ for $1\leq j\leq n$. In particular, since
$|\phi(z_j)|=c$, we see that
$(1-\delta,1+\delta)\phi(z_j)\subset\phi(z_j+\epsilon\D)$ for $1\leq
j\leq n$. Hence for each $s\in(1-\delta,1+\delta)$, we can pick
$w_1,\dots,w_n\in U$ such that $w_j\in z_j+\epsilon\D^m$ and
$\phi(w_j)=s\phi(z_j)$ for $1\leq j\leq n$. Then
$\phi(w_j)^k=s^k\phi(z_j)^k=(cs)^k$ and
$\sum\limits_{j=1}^n\lambda_jf(w_j)\in V$ since $w_j\in
z_j+\epsilon\D$. Hence $cs\in G$ for each $s\in(1-\delta,1+\delta)$
and therefore $c$ is an interior point of $G$. Since $c$ is an
arbitrary point of $G$, $G$ is open. This completes the proof of
(\ref{t3a}.2). It remains to apply Theorem~\ref{t3a} to conclude the
proof of Theorem~\ref{t4}.

\subsection{Proof of Corollary~\ref{co1}}

Note that $\H^*$ with the usual norm is a Banach space. Consider the
map $f:U\to \H^*$ defined by the formula $f(z)(x)=x(z)$. It is
straightforward to verify that $f$ is holomorphic,
$M_\phi^*f(z)=\phi(z)f(z)$ for each $z\in U$ and $\spann\{f(z):z\in
U\}$ is dense in $\H^*$. The latter is a consequence of the fact
that evaluation functionals separate points of $\H$. Using
Theorem~\ref{t4}, we immediately obtain that
$G_0=\uuu\{zM_\phi^*:b^{-1}<|z|<a^{-1}\}$ is a dense
$G_\delta$-subset of $\H^*$. Now consider the map $R:\H\to\H^*$,
$Rx(y)=\langle y,x\rangle$, where $\langle\cdot,\cdot\rangle$ is the
scalar product of the Hilbert space $\H$. According to the Riesz
theorem, $R$ is an $\R$-linear isometric isomorphism (it happens to
be complex conjugate linear). It is also easy to see that
$R^{-1}S^*R=S^\star$ for any $S\in L(\H)$, where $S^*$ is the dual
of $S$ and $S^\star$ is the Hilbert space adjoint of $S$. Hence
$G=R^{-1}(G_0)$, where $G=\uuu\{zM_\phi^\star:b^{-1}<|z|<a^{-1}\}$.
Since $R$ is a homeomorphism from $\H$ onto $\H^*$, $G$ is a dense
$G_\delta$-subset of $\H$.

\subsection{Proof of Corollary~\ref{co2}}

Consider the map $f:\C\to\H(\C)$ defined by the formula
$f(w)(z)=e^{wz}$. It is easy to see that $f$ is holomorphic,
$\spann\{f(z):z\in\C\}$ is dense in $\H(\C)$ and for each $w\in\C$,
$\ker(D-wI)=\spann\{f(w)\}$. In particular, $Df(w)=wf(w)$ and using
the equality $TD=DT$, we get $wTf(w)=DTf(w)$ for each $w\in \C$.
Hence $Tf(w)\in \ker(D-wI)=\spann\{f(w)\}$ for any $w\in\C$. Thus
there exists a unique function $\phi:\C\to\C$ such that
$Tf(w)=\phi(w)f(w)$ for each $w\in\C$. Using the fact that $f$ is
holomorphic and each $f(w)$ does not take value $0$, one can easily
verify that $\phi$ is holomorphic. Moreover, since $T$ is not a
scalar multiple of identity, $\phi$ is non-constant. By the Picard
theorem, any non-constant entire function takes all complex values
except for maybe one. Hence $\inf\limits_{w\in\C}|\phi(w)|=0$ and
$\sup\limits_{w\in\C}|\phi(w)|=\infty$. By Theorem~\ref{t4},
$\uuu\{zT:z\in\C^\star\}$ is a dense $G_\delta$-subset of $\H(\C)$.

\subsection{Proof of Corollary~\ref{co3}}

First, we consider the case $\K=\C$. Let $a<\alpha<\beta<b$. By the
assumptions, there is a dense subset $E$ of $X$ and a map $S:E\to E$
such that $TSx=x$, $\alpha^{-n}T^nx\to 0$ and $\beta^nS^nx\to 0$ for
each $x\in E$. Let $U=\{w\in\C:\alpha<|w|<\beta\}$. Since $X$ is
locally convex and complete, the relations $\alpha^{-n}T^nx\to 0$
and $\beta^nS^nx\to 0$ ensure that for each $w\in U$, the series
$\sum\limits_{n=1}^\infty w^{-n}T^nx$ and $\sum\limits_{n=1}^\infty
w^{n}S^nx$ converge in $X$ for any $x\in E$. Thus we can define
$$
u_{x,w}=x+\sum\limits_{n=1}^\infty (w^{-n}T^nx+w^{n}S^nx)\ \
\text{for $w\in U$ and $x\in E$}.
$$
Using the relations $TSx=x$ for $x\in E$ and $T\in L(X)$, one can
easily verify that $Tu_{x,w}=wu_{x,w}$ for each $x\in E$ and $w\in
U$. Now we consider
$$
F_{k,c}=\spann\{u_{x,w}:x\in E,\ w^k=c^k\}\ \ \text{for $k\in\N$ and
$c\in(\alpha,\beta)$}.
$$
We shall show that $F_{k,c}$ for $k\in\N$ and $c\in(\alpha,\beta)$
satisfy conditions (\ref{t3a}.1--\ref{t3a}.4). First, the equality
$Tu_{x,w}=wu_{x,w}$ implies that $T^ky=c^ky$ for any $y\in F_{k,c}$.
Hence (\ref{t3a}.1) is satisfied. Clearly $F_{k,c}\subseteq F_{m,c}$
whenever $k$ is a divisor of $m$. Hence for any $c\in
(\alpha,\beta)$ and any $k_1,\dots,k_n\in\N$, $F_{k_j,c}\subseteq
F_{k,c}$ for $1\leq j\leq n$, where $k=k_1\cdot{\dots}\cdot k_n$.
Thus (\ref{t3a}.4) is satisfied. It is easy to see that
$$
F_c=\bigcup_{k=1}^\infty F_{k,c}=\spann\{u_{x,w}:x\in E,\ w\in
c\G\},\ \ \text{where}\ \ \G=\{z\in\T:z^k=1\ \text{for some}\
k\in\N\}.
$$
In order to prove (\ref{t3a}.3), we have to show that $F_c$ is dense
in $X$. Assume the contrary. Since $F_c$ is a vector space and $X$
is locally convex, we can pick $g\in X^*$  such that $g\neq 0$ and
$g(y)=0$ for each $y\in F_c$. Hence for any $x\in E$ and $w\in c\G$,
we have $f_x(w)=0$, where $f_x(w)=g(u_{x,w})$. It is easy to verify
that for any $x\in E$, the function $f_x:U\to \C$ is holomorphic.
Since $f_x$ vanishes on $c\G$, the uniqueness theorem implies that
each $f_x$ is identically zero. On the other hand, the $0^{\rm th}$
Laurent coefficient of $f_x$ is $g(x)$. Hence $g(x)=0$ for any $x\in
E$. Since $E$ is dense in $X$, we get $g=0$. This contradiction
completes the proof of (\ref{t3a}.3). It remains to verify
(\ref{t3a}.2). Let $k\in\N$, $V$ be a non-empty open subset of $X$
and $G=\{c\in(\alpha,\beta):F_{k,c}\cap V\neq\varnothing\}$. We have
to show that $G$ is open in $\R$. Let $c\in G$. Then there exist
$x_1,\dots,x_n\in E$ and $w_1,\dots,w_n,\lambda_1,\dots,\lambda_n\in
\C$ such that $w_j^k=c^k$ for $1\leq j\leq n$ and
$\sum\limits_{j=1}^n\lambda_j u_{x_j,w_j}\in V$. Since for any fixed
$x\in E$, the map $w\mapsto u_{x,w}$ is continuous, there is
$\delta>0$ such that $y_s\in V$ if $|c-s|<\delta$, where
$y_s=\sum\limits_{j=1}^n\lambda_j u_{x_j,sw_j/c}$. On the other
hand, $y_s\in E_{k,s}$ for each $s$ and therefore
$(c-\delta,c+\delta)\cap(\alpha,\beta)\subseteq G$. Hence $c$ is an
interior point of $G$. Since $c$ is an arbitrary point of $G$, $G$
is open. This completes the proof of (\ref{t3a}.2). By
Theorem~\ref{t4}, $\uuu\{zT:\beta^{-1}<|z|<\alpha^{-1}\}$ is a dense
$G_\delta$-set whenever $a<\alpha<\beta<b$. Hence the set of common
hypercyclic vectors of the family $\{zT:b^{-1}<|z|<a^{-1}\}$ is a
dense $G_\delta$-subset of $X$ as a countable intersection of dense
$G_\delta$-sets. The proof of Corollary~\ref{co3} in the case
$\K=\C$ is complete.

Assume now that $\K=\R$. Let $X_\C=X\oplus iX$ and
$T_\C(u+iv)=Tu+iTv$ be complexifications of $X$ and $T$
respectively. It is straightforward to see that $T_\C$ satisfies the
same conditions with $E_\C=E+iE$ and $S_\C(u+iv)=Su+iSv$ taken as
$E$ and $S$. Corollary~\ref{co3} in the complex case implies that
$H_0=\uuu\{zT_\C:z\in\C,\ b^{-1}<|z|<a^{-1}\}$ is a dense
$G_\delta$-subset of $X_\C$. Clearly $H=\uuu\{zT:z\in\R,\
b^{-1}<|z|<a^{-1}\}$ contains the projection of $H_0$ onto $X$ along
$iX$ and therefore in dense in $X$. The fact that $H$ is a
$G_\delta$-subset of $X$ follows from Corollary~\ref{gc2}.

\section{Counterexamples on hypercyclic scalar multiples}

We find operators, whose existence is assured by Theorem~\ref{t5} in
the class of bilateral weighted shifts on $\ell_2(\Z)$. Recall that
if $w=\{w_n\}_{n\in\Z}$ is a bounded sequence of non-zero scalars,
then the unique $T_w\in L(\ell_2(\Z))$ such that $T_we_n=w_ne_{n-1}$
for $n\in\Z$, where $\{e_n\}_{n\in\Z}$ is the canonical orthonormal
basis of the Hilbert space $\ell_2(\Z)$, is called the {\it
bilateral weighted shift with the weight sequence} $w$.
Hypercyclicity of bilateral weighted shifts was characterized by
Salas \cite{sal}, whose necessary and sufficient condition is
presented in a more convenient shape in \cite{sh2}.

\begin{thmS}
Let $T_w$ be a bilateral weighted shift on $\ell_2(\Z)$. Then $T_w$
is hypercyclic if and only if for any $k\in\Z_+$,
\begin{equation}
\ilim\limits_{n\to\infty}(\ww(k-n+1,k)+ \ww(k+1,k+n)^{-1})=0,\ \
\text{where}\ \ww(a,b)=\smash{\prod_{j=a}^b}\,|w_j|\ \text{for}\
a,b\in\Z,\ a\leq b. \label{sal3}
\end{equation}
\end{thmS}

It is well-known and easy to see that a bilateral weighted shift
$T_w$ is invertible if and only if $\inf\limits_{n\in\Z}|w_n|>0$. In
this case condition (\ref{sal3}) can be rewritten in the following
simpler form.

\begin{thmSs}
Let $T_w$ be an invertible bilateral weighted shift on $\ell_2(\Z)$.
Then $T_w$ is hypercyclic if and only if
\begin{equation}
\ilim\limits_{n\to\infty}(\ww(-n,0)+\ww(0,n)^{-1})=0. \label{sal4}
\end{equation}
\end{thmSs}

\subsection{Proof of Theorem~\ref{t5}, Part II}

First, we prove few elementary lemmas. The following one generalizes
the fact that the set of hypercyclic vectors of a hypercyclic
operator is dense.

\begin{lemma}\label{el} Let $X$ be a topological vector space
and $\cal A$ be a family of pairwise commuting continuous linear
operators on $X$. Then the set $\uuu({\cal
A})=\bigcap\limits_{T\in{\cal A}}\uuu(T)$ is either empty or dense
in $X$.
\end{lemma}

\begin{proof} Let $x\in \uuu({\cal A})$ and $S\in{\cal A}$.
We have to show that $\uuu({\cal A})$ is dense in $X$. Since $x$ is
a hypercyclic vector for $S$, $O(S,x)=\{S^nx:n\in\Z_+\}$ is dense in
$X$ and therefore $S$ has dense range. Take any $T\in{\cal A}$.
Since $TS=ST$, $O(T,S^mx)=S^m(O(T,x))$ for each $m\in\Z_+$. Since
$x\in\uuu(T)$ and $S^m$ has dense range, $O(T,S^mx)$ is dense in
$X$. Hence $S^mx\in\uuu(T)$ for any $T\in{\cal A}$ and $m\in\Z_+$.
That is, $O(S,x)\subseteq \uuu({\cal A})$. Since $O(S,x)$ is dense
in $X$, so is $\uuu({\cal A})$.
\end{proof}

\begin{lemma} \label{inter} Let $X$ be a locally convex topological
vector space, $T\in L(X)$, $A\subseteq (0,\infty)$ and
$x\in\uuu\{cT:c\in\Lambda\}$. Assume also that there exists a
non-empty open subset $U$ of $X$ such that
\begin{equation}\label{qu}
\sum\limits_{n\in Q_U}n^{-1}<\infty,\quad\text{where}\quad
Q_U=\{n\in\N:a^nT^nx\in U\ \ \text{for some}\ \ a\in A\}.
\end{equation}
Then $A$ has zero Lebesgue measure.
\end{lemma}

\begin{proof} Clearly we can assume that $A\neq\varnothing$
and therefore $\Lambda\neq\varnothing$, where $\Lambda=\ln(A)=\{\ln
a:a\in A\}$. Since $X$ is Hausdorff and locally convex, we can find
a continuous seminorm $p$ on $X$ such that $V=U\cap \{u\in
X:1<p(u)<e\}$ is non-empty. It suffices to show that $\Lambda$ has
zero Lebesgue measure.  Let $\alpha\in\Lambda$ and $m\in\N$. Since
$x$ is hypercyclic for $e^{\alpha}T$ and $V$ is open, we can find
$n\geq m$ such that $e^{\alpha n}T^n\in V\subseteq U$. Then $n\in
Q_U$ and $p(e^{\alpha n}T^nx)\in (1,e)$. Hence
$$
\alpha\in (\alpha_n,\beta_n),\ \ \text{where}\ \
\alpha_n=\frac{-\ln(p(T^nx))}{n}\ \ \text{and}\ \
\beta_n=\frac{1-\ln(p(T^nx))}{n}.
$$
Since $\alpha\in \Lambda$ is arbitrary, we obtain
$$
\Lambda\subseteq\bigcup_{n\in Q_U,\ n\geq m}(\alpha_n,\beta_n)\ \
\text{for any $m\in\N$}.
$$
On the other hand, $(\alpha_n,\beta_n)$ is an interval of length
$n^{-1}$. Then (\ref{qu}) and the last display imply that $\Lambda$
can be covered by intervals with arbitrarily small sum of lengths.
That is, $\Lambda$ has zero Lebesgue measure.
\end{proof}

For $k\in\N$, we denote
\begin{equation}\label{III}
\begin{array}{l}\textstyle m_k=2^{3k^2},\ \ I_k^{-}=\{n\in\N:\frac78m_k\leq n<m_k\},\ \
I_k^{+}=\{n\in\N:m_k<n\leq \frac98m_k\}\\
\text{and}\ \ I_k=I_k^{-}\cup
I_k^{+}\cup\{m_k\}=\{n\in\N:\frac78m_k\leq n\leq
\frac98m_k\}.\end{array}
\end{equation}

Consider the sequence $w=\{w_n\}_{n\in\Z}$ defined by the formula
\begin{equation}\label{wp2}
w_n=\left\{\begin{array}{ll}2^8&\text{if}\ \ n\in I_k^-\cup-I_k^+,\
k\in\N\\ 2^{-8}&\text{if}\ \ n\in I_k^+\cup-I_k^-,\ k\in\N\\
1&\text{otherwise.}\end{array}\right.
\end{equation}
Clearly $w$ is a sequence of positive numbers and
$0<2^{-8}=\inf\limits_{n\in\Z}w_n<\sup\limits_{n\in\Z}w_n=2^8<\infty$.
Hence $T_w$ is an invertible bilateral weighted shift. In order to
prove Part~II of Theorem~\ref{t5} it is enough to verify the
following statement.

\begin{example}\label{p2t5} Let $w$ be the weight sequence defined
by $(\ref{wp2})$ and $T=T_w$  be the corresponding bilateral
weighted shift on $\ell_2(\Z)$. Then $M_T=(1/2,2)$ and any
$\Lambda\subseteq(1/2,2)$ has Lebesgue measure $0$ if the family
$\{aT:a\in\Lambda\}$ has a common hypercyclic vector.
\end{example}

\begin{proof} Using the definition (\ref{wp2}) of the sequence $w$, it is easy to
verify that for any $n\in\N$,
\begin{equation}\label{wwp2}
\beta(n)=\left\{\begin{array}{ll}2^{8n-7m_k+8}&\text{if}\ \ n\in
I_k^-,\ k\in\N,\\ 2^{9m_k-8n}&\text{if}\ \ n\in I_k^+,\ k\in\N,\\
1&\text{otherwise,}\end{array}\right. \qquad\text{where}\quad
\beta(n)=\prod_{j=0}^n w_j.
\end{equation}
Moreover, $w_n^{-1}=w_{-n}$ for any $n\in\Z$. Using this fact and
the equality $w_0=1$, we get
\begin{equation}\label{wwp21}
\ww(j,n)=\left\{\begin{array}{ll}\beta(n)\beta(j-1)^{-1}&\text{if}\
\ j\geq 1,\\ \beta(-1-n)\beta(-j)^{-1}&\text{if}\ \ n\leq -1,\\
\beta(n)\beta(-j)^{-1}&\text{if}\ j\leq 0,\ \text{and}\ n\geq
0\end{array}\right. \qquad\text{for any $j,n\in\Z$, $j\leq n$,}
\end{equation}
where the numbers $\ww(j,n)$ are defined in (\ref{sal3}). In
particular, $\ww(0,n)=\beta(n)$ and $\ww(-n,0)=\beta(n)^{-1}$ for
each $n\in\N$. This observation together with Theorem~${\rm S}'$ and
the fact that $aT=T_{aw}$ for $a\neq 0$ imply that for $a>0$,
\begin{equation}\label{hy1}
\text{$aT$ is hypercyclic if and only if}\quad
\ilim_{n\to\infty}\beta(n)^{-1}\bigl(a^n+a^{-n}\bigr)=0.
\end{equation}
By (\ref{wwp2}), $1\leq \beta(n)\leq 2^n$ for $n\in\N$, which
together with (\ref{hy1}) implies that $M_T\subseteq (1/2,2)$. On
the other hand, by (\ref{wwp2}), $\beta(m_k)=2^{m_k}$ for each
$k\in\N$. Hence $\beta(m_k)^{-1}\bigl(a^{m_k}+a^{-m_k}\bigr)\to 0$
as $k\to\infty$ for any $a\in(1/2,2)$. According to (\ref{hy1}),
$aT$ is hypercyclic if $1/2<a<2$. Hence $M_T=(1/2,2)$.

Let now $\Lambda$ be a non-empty subset of $(1/2,2)$ such that the
family $\{aT:a\in\Lambda\}$ has common hypercyclic vectors. We have
to demonstrate that $\Lambda$ has zero Lebesgue measure. Pick
$\epsilon>0$ such that $\frac{\epsilon}{1-\epsilon}<2^{-8}$. By
Lemma~\ref{el}, there is a common hypercyclic vector $x$ of the
family $\{aT:a\in\Lambda\}$ such that $\|x-e_{-1}\|<\epsilon$. Let
$$
\smash{Q=\{n\in\N:\|a^nT^nx-e_0\|<\epsilon\ \ \text{for some}\ \
a\in\Lambda\}\ \ \text{and}\ \ J=\bigcup_{k=1}^\infty I_k.}
$$
First, we show that $Q\subseteq J$. Let $n\in Q$. Then there is
$a\in\Lambda$ such that $\|a^nT^nx-e_0\|<\epsilon$. Hence
$$
\text{$|\langle a^nT^nx,e_0\rangle|>1-\epsilon$ \ and \ $|\langle
a^{n}T^nx,e_{-n-1}\rangle|<\epsilon$.}
$$
Using (\ref{wwp21}), we get $\langle
a^nT^nx,e_0\rangle=a^n\beta(n)x_n$ and $\langle
a^nT^nx,e_{-n-1}\rangle=a^n\beta(n)^{-1}x_{-1}$. Then from the last
display it follows that
$$
\text{$a^n\beta(n)|x_n|>1-\epsilon$ \ and \
$a^n\beta(n)^{-1}w_n|x_{-1}|<\epsilon$}.
$$
Since $\|x-e_{-1}\|<\epsilon$, $|x_{-1}|>1-\epsilon$ and
$|x_n|<\epsilon$. Then according to the last display,
$$
\beta(n)>\frac{1-\epsilon}{\epsilon}\max\{a^n,a^{-n}\}\geq
\frac{1-\epsilon}{\epsilon}>2^8>1.
$$
By (\ref{wwp2}), $\beta(j)=1$ if $j\notin J$. Hence $n\in J$. Since
$n$ is an arbitrary element of $Q$, we get $Q\subseteq J$.

Next, we show that $(Q-Q)\cap \N\subseteq J$. Indeed, let $m,n\in Q$
be such that $m>n$. Since $m,n\in Q$, we can pick $a,b\in\Lambda$
such that $\|a^nT^nx-e_0\|<\epsilon$ and $\|b^mT^mx-e_0\|<\epsilon$.
In particular,
$$
\text{$|a^nT^nx,e_0\rangle|>1-\epsilon$, $|\langle
b^mT^mx,e_0\rangle|>1-\epsilon$, $|\langle
a^nT^nx,e_{m-n}\rangle|<\epsilon$ and $|\langle
b^mT^mx,e_{n-m}\rangle|<\epsilon$}.
$$
Using (\ref{wwp21}), we get
$$
\begin{array}{ll}
\text{$\langle a^nT^nx,e_0\rangle=a^n\beta(n)x_n$},& \text{$\langle
a^nT^nx,e_{m-n}\rangle=
a^n\beta(m)\beta(m-n)^{-1}x_m$}, \\
\text{$\langle b^mT^mx,e_0\rangle=b^m\beta(m)x_m$}, &\text{$\langle
b^mT^mx,e_{n-m}\rangle=b^m\beta(n)\beta(m-n-1)^{-1}x_n$}.
\end{array}
$$
According to the last two displays,
$$
\beta(m-n-1)>\frac{1-\epsilon}{\epsilon}a^nb^{-m}\ \ \text{and}\ \
\beta(m-n)>\frac{1-\epsilon}{\epsilon}a^{-n}b^m.
$$
Since $\beta(m-n)=\beta(m-n-1)w_{m-n}\geq 2^{-8}\beta(m-n-1)$ from
the last display it follows that
$$
\beta(m-n)>2^{-8}\frac{1-\epsilon}{\epsilon}\max\{a^nb^{-m},a^{-n}b^m\}\geq
2^{-8}\frac{1-\epsilon}{\epsilon}>1.
$$
Since $\beta(j)=1$ if $j\notin J$, we have $m-n\in J$. Hence
$(Q-Q)\cap \N\subseteq J$.

Let now $k\in\N$ and $m,n\in Q\cap I_k$ be such that $m>n$. Since
$(Q-Q)\cap \N\subseteq J$, we have $m-n\in J$. Since $m,n\in I_k$,
we get $m-n\leq \frac{m_k}4<\frac{7m_k}{8}=\min I_k$ . Hence $m-n\in
\bigcup\limits_{j=0}^{k-1}I_j$, where $I_0=\varnothing$. Then
$|m-n|\leq \frac{9m_{k-1}}{8}<2m_{k-1}$, where $m_0=1$. Hence $Q\cap
I_k$ has at most $2m_{k-1}$ elements. On the other hand, $n\geq
\frac{7m_k}{8}\geq \frac{m_k}2$ for any $n\in I_k$ and therefore
$$
\sum_{n\in Q\cap I_k}n^{-1}\leq 2m_{k-1}
\frac{2}{m_k}=\frac{4m_{k-1}}{m_k}\leq 2^{-k},
$$
where the last inequality follows from the definition of $m_k$.
Since $Q\subseteq J$ and $J$ is the union of disjoint sets $I_k$, we
obtain
$$
\sum_{n\in Q}n^{-1}=\sum_{k=1}^\infty\sum_{n\in Q\cap
I_k}n^{-1}\leq\sum_{k=1}^\infty 2^{-k}=1<\infty.
$$
Using the definition of $Q$ and Lemma~\ref{inter}, we now see that
$\Lambda$ has zero Lebesgue measure.
\end{proof}

\subsection{Proof of Theorem~\ref{t5}, Part~I}

Consider the sequences $\{a_n\}_{n\in\Z}$ and $\{w_n\}_{n\in\Z}$
defined by the formulae
\begin{equation}\label{an}
a_n=\left\{
\begin{array}{ll}
1&\text{if $|n|\leq 5$ or $-2\cdot5^k\leq n<-5^k$}\\
&\text{or $-5^{k+1}\leq n<-4\cdot 5^k$, $k\in\N$,}\\
8^{-1}&\text{if $-3\cdot5^k\leq n<-2\cdot5^k$, $k\in\N$,}\\
8&\text{if $-4\cdot5^k\leq n<-3\cdot5^k$, $k\in\N$,}\\
2^{-1}&\text{if $2\cdot5^k<n\leq4\cdot5^k$, $k\in\N$,}\\
4^{-1}&\text{if $5^k<n\leq2\cdot5^k$, $k\in\N$,}\\
16&\text{if $4\cdot5^k<n\leq5^{k+1}$, $k\in\N$;}
\end{array}
\right. \qquad w_n=\left\{
\begin{array}{ll}
1&\text{if $|n|\leq 1$,}\\
n(n-1)^{-1}a_n&\text{if $n\geq 2$,}\\
(n+1)n^{-1}a_n&\text{if $n\leq -2$.}
\end{array}
\right.
\end{equation}
It is easy to see that $w$ is a bounded sequence of positive numbers
and $\inf\limits_{n\in\Z}w_n>0$. Hence the bilateral weighted shift
$T_w$ is invertible. In order to prove Part~I of Theorem~\ref{t5} it
is enough to verify the following statement.

\begin{example}\label{p1t5} Let $w$ be the weight sequence defined
by $(\ref{an})$ and $S=T_w$  be the corresponding bilateral weighted
shift on $\ell_2(\Z)$. Then $M_S=\{1,2\}$.
\end{example}

\begin{proof} Using (\ref{an}), one can
easily verify that
\begin{align}\label{gamma1}
\gamma_+(n)&=\left\{
\begin{array}{ll}
4^{5^k-n}&\text{if $5^k<n\leq2\cdot5^k$, $k\in\N$,}\\
2^{-n}&\text{if $2\cdot5^k<n\leq4\cdot5^k$, $k\in\N$,}\\
16^{n-5^{k+1}}&\text{if $4\cdot5^k<n\leq5^{k+1}$, $k\in\N$,}
\end{array}
\right.\qquad\qquad\quad\ \ \text{where}\quad
\gamma_+(n)=\prod_{j=0}^n a_j,
\\
\label{gamma2} \gamma_-(n)&=\left\{
\begin{array}{ll}
1&\text{if $5^k<n\leq2\cdot5^k$ or $4\cdot5^k<n\leq5^{k+1}$, $k\in\N$,}\\
8^{2\cdot 5^k-n}&\text{if $2\cdot5^k<n\leq3\cdot5^k$, $k\in\N$,}\\
8^{n-4\cdot 5^k}&\text{if $3\cdot5^k<n\leq4\cdot5^{k}$, $k\in\N$.}
\end{array}
\right.\!\!\!\!\!\!\!\!\text{where}\quad
\gamma_-(n)=\!\!\prod_{j=-n}^0 \!\!a_j.
\end{align}
For brevity we denote $\beta_+(n)=\ww(0,n)$ and
$\beta_-(n)=\ww(-n,0)$, where $\ww(k,l)$ are defined in
(\ref{sal3}). By definition of $w$,
\begin{equation}\label{bega}
\beta_+(n)=n\gamma_+(n)\quad\text{and}\quad\beta_-(n)=\frac{\gamma_-(n)}{n}\
\ \text{for any $n\in\N$}.
\end{equation}
According to (\ref{gamma1}) and (\ref{gamma2}),
$\gamma_+(5^k)=\gamma_-(5^k)=1$ and $\gamma_+(3\cdot
5^k)=\gamma_-(3\cdot 5^k)=8^{-5^k}$  for any $k\in\N$. Using
(\ref{bega}), we get $\beta_+(5^k)^{-1}=\beta_-(5^k)=5^{-k}\to 0$
and $(2^{3\cdot 5^k}\beta_+(3\cdot 5^k))^{-1}=2^{3\cdot
5^k}\beta_-(3\cdot 5^k)=3^{-1}5^{-k}\to 0$ as $k\to\infty$. Applying
Theorem~${\rm S}'$ to $S=T_w$ and $2S=T_{2w}$, we see that $S$ and
$2S$ are both hypercyclic.

Let $c>0$ be such that $cS=T_{cw}$ is hypercyclic. By Theorem~${\rm
S}'$, there exists a strictly increasing sequence $\{n_j\}_{j\in\N}$
of positive integers such that
\begin{equation}\label{limi}
(c^{n_j}\beta_+(n_j))^{-1}+c^{n_j}\beta_-(n_j)\to 0\ \ \text{as
$j\to\infty$}.
\end{equation}
Let $k_j$ be the integer part of $\log_5n_j$. Then $n_j=b_j5^{k_j}$,
where $1\leq b_j<5$. Passing to a subsequence, if necessary, we can
additionally assume that $b_j\to b\in[1,5]$ as $j\to\infty$. Using
(\ref{gamma1}) and (\ref{gamma2}), one can easily verify that
convergence of $b_j$ to $b$ implies that
\begin{equation}\label{li}
\lim_{j\to\infty}\gamma_+(n_j)^{1/n_j}=\lambda_+(b) \ \ \
\text{and}\ \ \ \lim_{j\to\infty}\gamma_-(n_j)^{1/n_j}=\lambda_-(b),
\end{equation}
where the continuous positive functions $\lambda_+$ and $\lambda_-$
on $[1,5]$ are defined by the formula
\begin{equation}\label{lam}
\lambda_+(b)=\left\{
\begin{array}{ll}
4^{b^{-1}-1}&\text{if $1\leq b<2$,}\\
1/2&\text{if $2\leq b\leq 4$,}\\
16^{1-5b^{-1}}&\text{if $4<b\leq5$}
\end{array}
\right. \quad\text{and}\quad \lambda_-(b)=\left\{
\begin{array}{ll}
1&\text{if $b\in[1,2]\cup[4,5]$,}\\
8^{2b^{-1}-1}&\text{if $2<b\leq 3$,}\\
8^{1-4b^{-1}}&\text{if $3<b<4$.}
\end{array}
\right.
\end{equation}
According to (\ref{bega}),
$$
\lim\limits_{n\to\infty}\biggl(\frac{\beta_+(n)}{\gamma_+(n)}\biggr)^{1/n}=1\
\ \text{and}\ \
\lim\limits_{n\to\infty}\biggl(\frac{\beta_-(n)}{\gamma_-(n)}\biggr)^{1/n}=1
$$
From (\ref{li}) and the above display it follows that
$$
\lim_{j\to\infty}\bigl(c^{n_j}\beta_+(n_j)^{1/n_j}\bigr)^{-1/n_j}=(c\lambda_+(b))^{-1}
\ \ \ \text{and}\ \ \
\lim_{j\to\infty}\bigl(c^{n_j}\beta_+(n_j)^{1/n_j}\bigr)^{1/n_j}=c\lambda_-(b).
$$
These equalities together with (\ref{limi}) imply that
$(c\lambda_+(b))^{-1}\leq 1$ and $c\lambda_-(b)\leq 1$. In
particular, $\frac{\lambda_-(b)}{\lambda_+(b)}\leq 1$. On the other
hand, (\ref{lam}) implies that $\frac{\lambda_-(b)}{\lambda_+(b)}>1$
for $b\in (1,3)\cup(3,5)$. Hence $b\in\{1,3,5\}$. If $b\in\{1,5\}$,
then $\lambda_-(b)=\lambda_+(b)=1$ and the inequalities
$(c\lambda_+(b))^{-1}\leq 1$ and $c\lambda_-(b)\leq 1$ imply that
$c\leq 1$ and $c^{-1}\leq 1$. That is, $c=1$. If $b=3$, then
$\lambda_-(b)=\lambda_+(b)=1/2$ and the inequalities
$(c\lambda_+(b))^{-1}\leq 1$ and $c\lambda_-(b)\leq 1$ imply that
$c/2\leq 1$ and $2/c\leq 1$. That is, $c=2$. Thus $c\in\{1,2\}$.
Hence $M_S=\{1,2\}$.
\end{proof}

\section{Proof of Theorem~\ref{t2}}

The main tool in the proof is the following result by Macintyre and
Fuchs. The following theorem is a part of Theorem~1 in \cite{mf}.

\begin{thmMF} Let $d>0$, $n\in\N$ and $z_1,\dots,z_n\in\C$. Then
there exist $n$ closed disks $D_1,\dots,D_n$ on the complex plane
such that their radii $r_1,\dots,r_n$ satisfy
$\sum\limits_{j=1}^nr_j^2\leq 4d^2$ and
\begin{equation}\label{MF}
\sum_{j=1}^n|z-z_j|^{-2}<\frac{n(1+\ln n)}{d^2}\ \ \text{for any}\ \
z\in\C\setminus\bigcup_{j=1}^n D_j.
\end{equation}
\end{thmMF}

We also need the following elementary lemma.

\begin{lemma}\label{EL} Let $X$ be a  topological vector
space, $T\in L(X)$ and $f\in X^*\setminus\{0\}$. Assume also that
there exist a polynomial $p$ such that $p(T)$ is hypercyclic. Then
the sequence $\{(T^*)^nf\}_{n\in\Z_+}$ is linearly independent.
\end{lemma}

\begin{proof} Assume that the sequence $\{(T^*)^nf\}_{n\in\Z_+}$ is
linearly dependent. Then we can pick $n\in\N$ such that $(T^*)^nf\in
L= \spann\{f,T^*f,\dots,(T^*)^{n-1}f\}$. It follows that $L$ is a
non-trivial finite dimensional invariant subspace for $T^*$. Hence
$L^\perp=\{x\in X:g(x)=0\ \ \text{for any $g\in L$}\}$ is a closed
linear subspace of $X$ of finite positive codimension invariant for
$T$. Clearly $L^\perp$ is also invariant for $p(T)$. We have
obtained a contradiction with a result of Wengenroth \cite{ww},
according to which hypercyclic operators on topological vector
spaces have no closed invariant subspaces of positive finite
codimension.
\end{proof}

We are ready to prove Theorem~\ref{t2}. Let $X$ be a complex
topological vector space such that $X^*\neq \{0\}$, $T\in L(X)$ and
$\Lambda$ be a non-empty subset of $\R\times \C$ for which the
family ${\cal A}=\{e^a(T+bI):(a,b)\in\Lambda\}$ has a common
hypercyclic vector. In order to prove Theorem~\ref{t2} it suffices
to show that $\Lambda$ has zero three dimensional Lebesgue measure.
Pick a non-zero $f\in X^*$. By Lemma~\ref{el}, the set $\uuu({\cal
A})$ of common hypercyclic vectors for operators from $\cal A$ is
dense in $X$. Since $\uuu({\cal A})$ is also closed under
multiplications by non-zero scalars, we can pick $x\in \uuu({\cal
A})$ such that $f(x)=1$. For each $n\in\N$ consider the complex
polynomial
\begin{equation}\label{pn1}
p_n(b)=f((T+bI)^nx)=\sum_{j=0}^n  \bin nj((T^*)^{n-j}f)(x)b^j.
\end{equation}
Clearly $p_n$ is a polynomial of degree $n$ with coefficient
$1=f(x)$ in front of $b^n$ (such polynomials are usually called {\it
monic}). Differentiating (\ref{pn1}) by $b$, we obtain that
$p'_n(b)=nf((T+bI)^{n-1}x)=np_{n-1}(b)$. That is,
\begin{equation}\label{pn2}
p'_n=np_{n-1}\ \ \text{for each}\ \ n\in\N.
\end{equation}
Applying (\ref{pn2}) twice, one can easily verify that
\begin{equation}\label{pn3}
\bigl(p'_n/p_n\bigr)'=n^2\biggl(\Bigl(
1-\frac1n\Bigr)\frac{p_{n-2}}{p_n}-\Bigl(\frac{p_{n-1}}{p_n}\Bigr)^2\biggr)\
\ \text{for each}\ \ n\geq 2.
\end{equation}
The equality (\ref{pn3}) immediately implies the following
inequality:
\begin{equation}\label{pn4}
\bigl|(p'_n/p_n)'\bigr|\geq
n^2\biggl(\Bigl|\frac{p_{n-2}}{2p_n}\Bigr|-\Bigl|\frac{p_{n-1}}{p_n}\Bigr|^2\biggr)\
\ \text{for each}\ \ n\geq 2.
\end{equation}

\begin{lemma}\label{qq1} For any $(a,b)\in\Lambda$ and $k\in\Z_+$,
the sequence $\{v_n\}_{n\geq k}$ is dense in $\C^{k+1}$, where
$v_n=e^{an}(p_n(b),p_{n-1}(b),\dots,p_{n-k}(b))$.
\end{lemma}

\begin{proof} Assume the contrary. Then there exist
$(a,b)\in\Lambda$ and a non-empty open subset $W$ of $\C^{k+1}$ such
that $v_n\notin W$ for each $n\geq k$. Let $S=e^a(T+bI)$. By
definition of $p_m$, for $0\leq j\leq k$,
$$
e^{an}p_{n-j}(b)=e^{an}f((T+bI)^{n-j}x)=e^{aj}f(S^{n-j}x)=e^{aj}
(S^*)^{k-j}f(S^{n-k}x).
$$
Thus the relation $v_n\notin W$ can be rewritten as $S^{n-k}x\notin
R^{-1}(W)$, where the linear operator $R:X\to \C^{k+1}$ is defined
by the formula
$$
(Ry)_l=e^{a(l-1)}(S^*)^{k-l+1}f(y)\ \ \text{for}\ \ 1\leq l\leq k+1.
$$
By Lemma~\ref{EL}, continuous linear functionals
$f,S^*f,\dots,(S^*)^kf$ are linearly independent. It follows that
$R$ is continuous and surjective. Hence $V=R^{-1}(W)$ is a non-empty
open subset of $X$. Thus $S^{n-k}x$ does not meet the non-empty open
set $V$ for each $n\geq k$, which is impossible since $x\in\uuu(S)$.
\end{proof}

By Lemma~\ref{qq1} with $k=2$, for any $(a,b)\in\Lambda$, the
sequence $\{v_n=e^{an}(p_n(b),p_{n-1}(b),p_{n-2}(b))\}_{n\geq 2}$ is
dense in $\C^3$. Since the map $F:\C^\star\times\C^2\to\C^3$,
$F(u,v,w)=(u,v/u,w/u)$ is continuous and has dense range,
$\{F(u_n):n\geq 2,\ p_n(b)\neq 0\}$ is dense in $\C^3$. That is,
$$
\text{$\{(e^{an}p_n(b),p_{n-1}(b)/p_n(b),p_{n-2}(b)/p_n(b)):n\geq
2,\ p_n(b)\neq 0\}$ is dense in $\C^3$.}
$$
It follows that any $(a,b)\in\Lambda$ is contained in infinitely
many sets $C_n$, where
$$
C_n=\{(a,b)\in\R\times\C:1<|e^{an}p_n(b)|<e,\
|p_{n-1}(b)/p_n(b)|<1,\ |p_{n-2}(b)/p_n(b))|>8\}.
$$
That is,
\begin{equation}\label{l1}
\Lambda\subseteq\Lambda^*=\bigcap_{m=1}^\infty\bigcup_{n\geq m}C_n.
\end{equation}
Clearly, $C_n\subseteq \R\times B_n$, where
$$
B_n=\{b\in\C:|p_{n-1}(b)/p_n(b)|<1,\ |p_{n-2}(b)/p_n(b))|>8\}.
$$
Applying the inequality (\ref{pn4}), we see that
\begin{equation}\label{bn}
B_n\subseteq B'_n=\Bigl\{b\in\C:\bigl|(p'_n(b)/p_n(b))'\bigr|\geq
3n^2\Bigr\}.
\end{equation}
Since $p_n$ is a monic polynomial of degree $n$, there exist
$z_1,\dots,z_n\in\C$ such that
$$
p_n(b)=\prod_{j=1}^n (b-z_j)\ \ \ \text{and therefore}\ \ \
(p'_n(b)/p_n(b))'=-\sum_{j=1}^n (b-z_j)^{-2}.
$$
By Theorem~MF with $d=n^{-1/3}$, there are $n$ closed disks
$D_1,\dots,D_n$ on the complex plane such that their radii
$r_1,\dots,r_n$ satisfy
$$
\sum\limits_{j=1}^nr_j^2\leq 4n^{-2/3}\ \ \text{and}\ \
\bigl|(p'_n(b)/p_n(b))'\bigr|\leq
\sum_{j=1}^n|b-z_j|^{-2}\!\!<n^{5/3}(1+\ln n)\ \ \text{for any}\ \
b\in\C\setminus\bigcup\limits_{j=1}^n D_j .
$$
Since $n^{5/3}(1+\ln n)\leq 3n^2$ for any $n\in\N$, we see that
$B'_n\subseteq \bigcup\limits_{j=1}^n D_j$. Hence
$$
\mu_2(B_n)\leq \mu_2(B'_n)\leq \pi\sum\limits_{j=1}^nr_j^2\leq 4\pi
n^{-2/3},
$$
where $\mu_k$ is the $k$-dimensional Lebesgue measure. For each
$b\in B_n$, $A_{b,n}=\{a\in\R:(a,b)\in C_n\}$ can be written as
$$
A_{b,n}=\{a\in\R:1<|e^{an}p_n(b)|<e\}=\Bigl(\frac{-\ln|p_n(b)|}{n},
\frac{1-\ln|p_n(b)|}{n}\Bigr),
$$
which is an interval of length $n^{-1}$. Hence
$\mu_1(A_{b,n})=n^{-1}$ for each $b\in B_n$. By the Fubini theorem,
$$
\mu_3(C_n)=\int_{B_n}\mu_1(A_{b,n})\mu_2(db)=\frac{\mu_2(B_n)}{n}\leq
4\pi n^{-5/3}.
$$
According to (\ref{l1}) and the above estimate, we obtain
$$
\mu_3(\Lambda^*)\leq \inf_{m\in\N}4\pi\sum_{n=m}^\infty n^{-5/3}=0\
\ \text{since}\ \  \sum_{n=1}^\infty n^{-5/3}<\infty.
$$
Thus $\mu_3(\Lambda^*)=0$ and therefore $\mu_3(\Lambda)=0$ since
$\Lambda\subseteq \Lambda^*$. The proof of Theorem~\ref{t2} is
complete.

\section{Concluding remarks and open problems}

Lemma~\ref{EL} implies the following easy corollary.

\begin{corollary}\label{1212} Let $X$ be a topological vector
space such that $0<\dim X^*<\infty$. Then $X$ supports no
hypercyclic operators.
\end{corollary}

\begin{proof} Assume that $T\in L(X)$ is hypercyclic and $f\in X^*$,
$f\neq 0$. By Lemma~\ref{EL}, the sequence $\{(T^*)^nf\}_{n\in\Z_+}$
is linearly independent, which contradicts the inequality $\dim
X^*<\infty$.
\end{proof}

In particular, $\F$-spaces $X=L_p[0,1]\times \K^n$ for $0<p<1$ and
$n\in\N$ support no hypercyclic operators. Indeed, the dual of $X$
is $n$-dimensional. On the other hand, each separable infinite
dimensional Fr\'echet space supports a hypercyclic operator
\cite{bonper} and there are separable infinite dimensional
$\F$-spaces \cite{kal} that support no continuous linear operators
except the scalar multiples of $I$ and therefore support no
hypercyclic operators. However the following question remains open.

\begin{question}\label{q4} Let $X$ be a separable $\F$-space
such that $X^*$ is infinite dimensional. Is it true that there
exists a hypercyclic operator $T\in L(X)$?
\end{question}

Part~I of Theorem~\ref{t5} shows that there exists a continuous
linear operator $S$ on $\ell_2$ such that $M_S=\{1,2\}$, where
$M_S=\{a>0:aS\ \text{is hypercyclic}\}$. Using the same basic idea
as in the proof of Theorem~\ref{t5}, one can construct an invertible
bilateral weighted shift $S$ on $\ell_2(\Z)$ such that $M_S$ is a
dense subset of an interval and has zero Lebesgue measure. In
particular, $M_S$ and its complement are both dense in this
interval. It is also easy to show that for any $\F$-space $X$ and
any $T\in L(X)$, $M_T$ is a $G_\delta$-set. If $X$ is a Banach
space, then $M_T$ is separated from zero by the number $\|T\|^{-1}$.
These observations naturally lead to the following question.

\begin{question}\label{q0} Characterize subsets $A$ of $\R_+$ for
which there is $S\in L(\ell_2)$ such that $A=M_S$. In particular, is
it true that for any $G_\delta$-subset $A$ of $\R_+$ such that $\inf
A>0$, there exists $S\in L(\ell_2)$ for which $A=M_S$?
\end{question}

In the proof of Part~II of Theorem~\ref{t5} we constructed an
invertible bilateral weighted shift $T$ on $\ell_2(\Z)$ such that
$M_T=(1/2,2)$ and any subset $A$ of $(1/2,2)$ such that the family
$\{aT:a\in A\}$ has a common hypercyclic vector must be of zero
Lebesgue measure. It is also easy to see that our $T$ enjoys the
following extra property. Namely, if $E=\spann\{e_n:n\in\Z\}$ and
$x\in E$, then for $1/2<\alpha<\beta<2$, we have
$\alpha^{-m_k}T^{m_k} x\to 0$ and $\beta^{m_k}T^{-m_k}x\to 0$ with
$m_k=2^{3k^2}$. This shows that the convergence to zero condition in
Corollary~\ref{co3} can not be replaced by convergence to 0 of a
subsequence. Note that, according to the hypercyclicity criterion
\cite{bp}, the latter still implies hypercyclicity of all relevant
scalar multiples of $T$.

Recall that for $0<s\leq 1$ the {\it Hausdorff outer measure}
$\mu_s$ on $\R$ is defined as
$\mu_s(A)=\lim\limits_{\delta\downarrow0}\mu_{s,\delta}(A)$ with
$\mu_{s,\delta}(A)=\inf\sum(b_j-a_j)^s$, where the infimum is taken
over all sequences $\{(a_j,b_j)\}$ of intervals of length
$\leq\delta$, whose union contains $A$. The number
$\inf\{s\in(0,1]:\mu_s(A)=0\}$ is called the {\it Hausdorff
dimension} of $A$ . With basically the same proof Lemma~\ref{inter}
can be strengthened in the following way.

\begin{lemma} \label{inter1} Let $X$ be a locally convex topological
vector space, $T\in L(X)$, $s\in(0,1]$, $A\subseteq (0,\infty)$ and
$x$ be a common hypercyclic vector for the family
$\{cT:c\in\Lambda\}$. Assume also that there exists a non-empty open
subset $U$ of $X$ such that $\smash{\sum\limits_{n\in
Q_U}n^{-s}<\infty}$, where $Q_U$ is defined in $(\ref{qu})$. Then
$\mu_s(A)=0$.
\end{lemma}

Using Lemma~\ref{inter1} instead of Lemma~\ref{inter}, one can
easily see that the operator $T$ constructed in the proof of Part~II
of Theorem~\ref{t5} has a stronger property. Namely, any
$A\subset\R_+$ such that the family $\{cT:c\in A\}$ is hypercyclic
has zero Hausdorff dimension.

Theorem~CMP guarantees existence of common hypercyclic vectors for
all non-identity operators of a universal strongly continuous
semigroup $\{T_t\}_{t\geq 0}$ on an $\F$-space. On the other hand,
Theorem~CS shows that the non-identity elements of the 2-parametric
translation group on $\H(\C)$ have a common hypercyclic vector. The
latter group enjoys the extra property of depending holomorphically
on the parameter. Note that Theorem~\ref{t1} strengthens this
result.

\begin{question}\label{q1} Let $X$ be a complex Fr\'echet space and
$\{T_z\}_{z\in\C}$ be a holomorphic strongly continuous operator
group. Assume also that for each $z\in\C^\star$, the operator $T_z$
is hypercyclic. Is it true that the family $\{T_z:z\in\C^\star\}$
has a common hypercyclic vector?
\end{question}

\begin{question}\label{q1a} Let $X$ be a complex Fr\'echet space and
$\{T_z\}_{z\in\C}$ be a holomorphic strongly continuous operator
group. Assume also that for each $z,a\in\C^\star$, the operator
$aT_z$ is hypercyclic. Is it true that the family
$\{aT_z:a,z\in\C^\star\}$ has a common hypercyclic vector?
\end{question}

An affirmative answer to the following question would allow to
strengthen Theorem~\ref{t4}.

\begin{question}\label{q5} Let $T$ be a continuous linear operator
on a complex separable Fr\'echet space $X$ and $0\leq
a<b\leq\infty$. Assume also that for any $\alpha\in(a,b)$, the sets
$$
E_{\alpha}=\spann\Biggl(\bigcup_{|z|<\alpha}\ker(T-zI)\Biggr)\quad
\text{and}\quad
F_{\alpha}=\spann\Biggl(\bigcup_{|z|>\alpha}\ker(T-zI)\Biggr)
$$
are both dense in $X$. Is it true that the family
$\{zT:b^{-1}<|z|<a^{-1}\}$ has common hypercyclic vectors?
\end{question}

It is worth noting that according to the Kitai Criterion for $T$
from the above question, $zT$ is hypercyclic for any $z\in\C$ with
$b^{-1}<|z|<a^{-1}$. It also remains unclear whether the natural
analog of Theorem~\ref{t2} holds in the case $\K=\R$. For instance,
the following question is open.

\begin{question}\label{q3} Does there exist a continuous linear operator
$T$ on a real Fr\'echet space such that the family $\{aT+bI:a>0,\
b\in\R\}$ has a common hypercyclic vector?
\end{question}

{\bf Acknowledgements.} \ The author would like to thank Richard
Aron for interest and helpful comments.

\small\rm

\vskip1truecm

\scshape

\noindent Stanislav Shkarin

\noindent Queens's University Belfast

\noindent Department of Pure Mathematics

\noindent University road, Belfast, BT7 1NN, UK

\noindent E-mail address: \qquad {\tt s.shkarin@qub.ac.uk}


\begin{thebibliography}{99}

\itemsep=-2pt

\bibitem{ag}E.~Abakumov and J.~Gordon, \it Common hypercyclic vectors
for multiples of backward shift, \rm J. Funct. Anal. \bf200\rm
(2003), 494--504

\bibitem{bay}F.~Bayart, \it Common hypercyclic vectors for composition
operators.  J. Operator Theory  \bf52\rm\  (2004), 353--370

\bibitem{bagr}F.~Bayart and S.~Grivaux, \it Hypercyclicity and unimodular
point spectrum.  J. Funct. Anal. \bf226\rm\ (2005), 281--300

\bibitem{bm}F.~Bayart and \'E.~Matheron, \it How to get common universal
vectors, \rm Indiana Univ. Math. J. \bf56\rm\ (2007), 553--580

\bibitem{bp}J.~B\'es and A.~Peris, \it Hereditarily hypercyclic
operators, \rm J. Funct. Anal.  \bf167\rm\  (1999),  94--112

\bibitem{bonper}J.~Bonet and A.~Peris, \it Hypercyclic operators on
non-normable Fr\'echet spaces, \rm J. Funct. Anal. \bf159\rm\
(1998), 587--595

\bibitem{bosh}P.~Bourdon and J.~Shapiro, \it Spectral synthesis and common
cyclic vectors, \rm Michigan Math. J. \bf37\rm\ (1990), 71--90

\bibitem{chs}K.~Chan and R.~Sanders, \it Common supercyclic vectors for a
path of operators, \rm J. Math. Anal. Appl. \bf337\rm\ (2008),
646--658

\bibitem{semi}J.~Conejero, V.~M\"uller and A.~Peris, \it
Hypercyclic behaviour of operators in a hypercyclic $C\sb
0$-semigroup, \rm J. Funct. Anal. \bf244\rm\ (2007), 342--348

\bibitem{cm}G.~Costakis and P.~Mavroudis, \it Common hypercyclic entire
functions for multiples of differential operators, \rm Colloq. Math.
\bf111\rm\ (2008), 199--203

\bibitem{cs}G.~Costakis and M.~Sambarino, \it Genericity of
wild holomorphic functions and common hypercyclic vectors, \rm Adv.
Math. \bf182\rm\ (2004), 278--306

\bibitem{gp}E.~Gallardo-Guti\'errez and J.~Partington, \it Common
hypercyclic vectors for families of operators, \rm Proc. Amer. Math.
Soc. \bf136\ \rm(2008), 119--126

\bibitem{gosh}G.~Godefroy and J.~Shapiro, \it
Operators with dense, invariant, cyclic vector manifolds, \rm  J.
Funct. Anal. \bf98\rm\ (1991), 229--269

\bibitem{gri}S.~Grivaux, \it Hypercyclic operators, mixing operators
and the bounded steps problem, \rm J. Operator Theory \bf54\rm\
(2005), 147--168

\bibitem{ge1}K.~Grosse-Erdmann, \it Universal families and
hypercyclic operators, \rm Bull. Amer. Math. Soc. \bf36\rm\ (1999),
345--381

\bibitem{ge2}K.~Grosse-Erdmann, \it Recent developments in
hypercyclicity, \rm RACSAM Rev. R. Acad. Cienc. Exactas Fis. Nat.
Ser. A Mat. \bf 97\rm\  (2003), 273--286

\bibitem{kitai}C.~Kitai, \it Invariant closed sets for linear
operators, \rm Thesis, University of Toronto, 1982

\bibitem{levin}B.~Levin, \it Distribution of zeros of entire
functions, \rm AMS Providence, Rhode Island, 1980

\bibitem{muller}F.~Le\'on-Saavedra and Vladimir M\"uller, \it
Rotations of hypercyclic and supercyclic operators, \rm Integral
Equations and Operator Theory \bf50\rm\ (2004), 385--391

\bibitem{kal}N.~Kalton, N.~Peck and J.~Roberts, \it An $F$-space
sampler, \rm London Mathematical Society Lecture Note Series
\bf89\rm, Cambridge University Press, Cambridge, 1984

\bibitem{mf}A.~Macintyre and W.~Fuchs, \it Inequalities for the
logarithmic derivatives of a polynomial, \rm J. London Math. Soc.
\bf15\rm\ (1940), 162--168

\bibitem{sal}H.~Salas, \it Hypercyclic weighted shifts, \rm
Trans. Amer. Math. Soc. \bf347\rm\ (1995), 993--1004

\bibitem{sch}H.~Schaefer, \it Topological vector spaces, \rm
Springer, Berlin, 1971

\bibitem{sh1} S.~Shkarin, \it Universal elements for non-linear operators
and their applications, \rm J. Math. Anal. Appl. \bf348\rm\ (2008),
193--210

\bibitem{sh2}S.~Shkarin, \it Non-sequential weak supercyclicity and
hypercyclicity, \rm J. Funct. Anal. \bf242\rm\ (2007), 37--77

\bibitem{ww}J.~Wengenroth, \it Hypercyclic operators on non-locally
convex spaces, \rm Proc. Amer. Math. Soc. \bf131\rm\ (2003),
1759--1761


\end{thebibliography}
\end{document}